\documentclass[10pt]{amsart}
\calclayout
\usepackage{amsmath,amssymb,bbm}
\usepackage{enumerate}
\usepackage{color}
\usepackage{hyperref}
\usepackage[numbers]{natbib}
\usepackage{graphicx}
\usepackage{dsfont}
\usepackage{comment}
\usepackage{enumitem}
\usepackage{algorithm}

\usepackage{tikz}
\usetikzlibrary{arrows.meta,positioning}

\newtheorem{theorem}{Theorem}[section]
\newtheorem{proposition}[theorem]{Proposition}
\newtheorem{lemma}[theorem]{Lemma}
\newtheorem{corollary}[theorem]{Corollary}
\newtheorem{definition}[theorem]{Definition}
\newtheorem{remark}[theorem]{Remark}
\newtheorem{example}[theorem]{Example}

\newcommand{\E}{\mathbb{E}}

\newcommand{\F}{\mathbb{F}}
\newcommand{\N}{\mathbb{N}}

\newcommand{\R}{\mathbb{R}}
\renewcommand{\P}{\mathbb{P}}

\newcommand{\UU}{\bold{U}}

\newcommand{\XX}{\mathbf{X}}
\newcommand{\YY}{\mathbf{Y}}

\newcommand{\cP}{\mathcal{P}}

\newcommand{\cC}{\mathcal{C}}

\newcommand{\cE}{\mathcal{E}}

\newcommand{\sR}{\mathsf{R}}

\newcommand{\cU}{\mathcal{U}}

\newcommand{\cW}{\mathcal{W}}
\newcommand{\cAW}{\mathcal{AW}}
\newcommand{\cIW}{\mathcal{CW}}
\newcommand{\cX}{\mathcal{X}}
\newcommand{\cY}{\mathcal{Y}}
\newcommand{\dW}{\overrightarrow{\mathcal{W}}}

\providecommand{\keywords}[1]{\textbf{Keywords } #1}
\newcommand{\eqd}{\stackrel{\mathrm{d}}=}

\newcommand{\1}{\mathbf{1}}

\newcommand{\de}{\mathrm{d}}

\newcommand{\rank}{\mathsf{rank}}

\newcommand{\JA}[1]{{\color{blue} #1}}

\title[Dependence Measures via Adapted Optimal Transport]{Dependence Measures via Adapted Optimal Transport: Stability and Rates of Convergence}
\author{Jonathan Ansari}
\address{Jonathan Ansari\newline
\hbox{}\hspace{0.33cm} University of Salzburg \newline
\hbox{}\hspace{0.33cm} Department of Mathematics}
\email{jonathan.ansari@plus.ac.at}

\author{Johannes Wiesel}
\address{Johannes Wiesel\newline
\hbox{}\hspace{0.33cm} University of Copenhagen \newline
\hbox{}\hspace{0.33cm} Department of Mathematics}
\email{wiesel@math.ku.dk}
\date{\today}

\begin{document}

\begin{abstract}
    Recently studied dependence measures, such as Chatterjee's rank correlation, that characterize both independence and perfect functional dependence, provide a powerful framework for detecting nonlinear dependencies. However, these measures cannot be weakly continuous, which limits the applicability of classical plug-in estimators based on empirical distributions. This obstruction is natural, as such measures are defined via conditional distributions and not through their joint law alone.\\
    In this paper, we introduce an optimal transport-based mode of convergence that captures weak convergence of conditional distributions and restores continuity for a broad class of dependence measures. We relate this mode of convergence to the adapted Wasserstein distance, the Knothe-Rosenblatt distance and the $d_1$-metric on copulas. Building on this perspective, we propose a copula estimator based on the adapted empirical measure and compare it with the classical rank-based checkerboard estimator. For both estimators, we derive \(O(N^{-1/3})\)-rates of convergence with respect to metrics that capture conditional weak continuity. As a consequence,  we obtain the same rates for plug-in estimators of several classes of dependence measures, including rank-based and rearranged dependence measures.
\end{abstract}


\maketitle

\section{Introduction and main results}

Quantifying dependence between random variables is a fundamental problem in statistics, with implications ranging from theoretical probability to data-driven inference. In recent years, the theory of dependence measures has seen significant progress, with increasing attention given to measures that capture complex, non-monotone relationships and remain meaningful beyond the classical linear or monotone setting; see
\cite{chatterjee2023, catalano2025measures, Figalli-2025} for an overview.

Literature has focused on dependence measures \(\kappa\) on the space of probability measures \(\mathcal{P}(\cX\times \cY)\) that satisfy the following axioms, assuming \(\cX\) and \(\cY\) are Polish spaces:
\begin{enumerate}[label = (\Roman*)]
    \item \label{axiom1} \(\kappa(\pi) \in [0,1]\) for all \(\pi\in \mathcal{P}(\cX\times \cY)\),
    \item \label{axiom2} \emph{Zero-independence:} \(\kappa(\pi) = 0\) if and only if \(\pi = \pi_1 \otimes \pi_2\),
    \item \label{axiom3} \emph{Max-functionality:} \(\kappa(\pi) = 1\) if and only if \(\pi=(x, f(x))_{\#}\pi_1\) for some measurable function \(f\colon \cX \to \cY\).
\end{enumerate}
In the language of random variables \((X,Y)\colon (\Omega,\F,\P) \to (\cX\times\cY)\) with distribution \(\pi\), the value \(\kappa(\pi) = 0\) characterizes independence of \(X\) and \(Y\). Conversely, \(\kappa(\pi) = 1\) corresponds to \emph{perfect dependence} of \(Y\) on \(X\), i.e.~there exists a measurable function \(f\) such that \(Y = f(X)\) almost surely. Note that \(f\) is not assumed to be monotone. In the above, we implicitly assume that $\pi_2 = Y_\#\mathbb{P}$ is non-degenerate, thereby 
ruling out the pathological case where $Y$ is both independent of and 
perfectly dependent on $X$.

As pointed out by \citet[Corollary 1.1]{Buecher-2024}, a dependence measure \(\kappa\) that fulfills Axioms \ref{axiom2} and \ref{axiom3} cannot satisfy
\begin{enumerate}[label = (\Roman*)]\setcounter{enumi}{3}
    \item \emph{Weak continuity:} %If \((X_n,Y_n)\) is a sequence of random vectors weakly converging to \((X,Y)\), then \(\lim_{n\to \infty}\kappa(X_n,Y_n) = \kappa(X,Y)\).
    If \((\pi^n)\) is a sequence of distributions weakly converging to \(\pi\), then \(\lim_{n\to \infty}\kappa(\pi^n) = \kappa(\pi)\).
\end{enumerate}
The lack of weak continuity has several drawbacks. First, the empirical distribution function cannot serve as a basis for a consistent plug-in estimator. Second, dependence measures satisfying \ref{axiom2} and \ref{axiom3} lack robustness under slight model misspecifications with respect to the weak topology. Furthermore, independence tests based on these measures may suffer from trivial power against sequences of alternatives that converge weakly to the independent coupling.
Thus, to nevertheless establish robustness and convergence guarantees, a different notion of continuity is needed which accounts for weak convergence of conditional distributions. In this paper, we introduce such a notion of continuity for dependence measures and exploit this notion for stability results. Moreover, for the checkerboard estimator and an alternative copula estimator that we construct via the adapted empirical measure in \cite{Backhoff-Bartl-Beiglboeck-Wiesel-2022}, we determine rates of convergence that translate directly to dependence measures.

\subsection{A mode of convergence for dependence measures}

As a first main contribution, we provide a stronger mode of convergence that yields continuity results for dependence measures \(\kappa\) satisfying Axioms \ref{axiom1}--\ref{axiom3}. More precisely, 
we define the \emph{conditional Wasserstein} (pseudo-)distance
\begin{align}\label{eq:causal}
\mathcal{CW}(\pi,\tilde{\pi}) :=\inf_{\gamma\in \Pi(\pi_1, \tilde{\pi}_1)} \int \mathcal{W}(\pi_{x}, \tilde{\pi}_{\tilde x})\,\gamma(dx,d\tilde x)
\end{align} 
for probability measures $\pi, \tilde{\pi}\in \mathcal{P}(\cX\times \cY)$. Here, $\Pi(\pi_1, \tilde{\pi}_1)$ is the set of probability measures on $\mathcal{X}\times \mathcal{X}$ with marginals $\pi_1$ and $\tilde{\pi}_1$, and \(\mathcal{W}(\pi_{x}, \tilde{\pi}_{\tilde x})\) denotes the $1$-Wasserstein distance between the conditional laws \(\pi_{x}\) and \(\tilde{\pi}_{\tilde x}\) obtained via disintegration from \(\pi = \pi_1\otimes \pi_x\) and \(\tilde{\pi} = \tilde{\pi}_1\otimes \tilde{\pi}_{\tilde x}\) respectively; see \eqref{eq:disintegration}. Note that, in particular, the conditional Wasserstein distance is upper bounded by the adapted Wasserstein distance: 
\begin{align}\label{IW_leq_AW}
\mathcal{CW}(\pi,\tilde{\pi}) \le \mathcal{AW}(\pi, \tilde\pi):= \inf_{\gamma\in \Pi(\pi_1, \tilde{\pi}_1)} \int [ d(x,\tilde x)+\mathcal{W}(\pi_{x}, \tilde{\pi}_{\tilde x})]\,\gamma(dx,d\tilde x),
\end{align}
where \(d\) is a complete metric that metrizes the Polish space \(\cX\); see \eqref{eq:AW_simple}.
We show that many recently studied dependence measures that characterize independence and perfect dependence satisfy
\begin{enumerate}[label = (\Roman*')]\setcounter{enumi}{3}
    \item \emph{Conditional Wasserstein continuity: }%Adapted weak continuity \JA{how to call it? I think this is a better choice than conditional weak continuity}:} \textcolor{red}{Relaxed adapted weak continuity, integrated/iterated weak continuity? \JA{some more choices: iterated Wasserstein continuity, }} %If \((X_n,Y_n)\) is a sequence of random vectors converging to \((X,Y)\) with respect to \(\cIW\), then \(\lim_{n\to \infty}\kappa(X_n,Y_n) = \kappa(X,Y)\).
    If \((\pi^n)\) is a sequence of distributions converging to \(\pi\) with respect to \(\cIW\), then \(\lim_{n\to \infty}\kappa(\pi^n) = \kappa(\pi)\).
\end{enumerate}
Important examples are rank-based dependence measures \cite{chatterjee2021,chatterjee2020,gamboa2020}, Wasserstein correlations \cite{wiesel2022}, rearranged dependence measures \cite{strothmann2022}, kernel-based dependence measures \cite{deb2020,deb2020b}, and measures of sensitivity \cite{Ansari-LFT-2023}.
All these dependence measures quantify the strength of functional dependence of \(Y\) on \(X\) via the variability of conditional distributions.
A general class of functionals examined in \cite{Figalli-2025} satisfying Axioms \ref{axiom1}--\ref{axiom3} is of the form 
\begin{align}\label{kappa_figalli}
    \kappa(\pi) = \int \Psi(\pi_2,\pi_x)\, \pi_1(\de x),
\end{align}
where $\pi_2$ denotes the second marginal of $\pi.$
Under a Lipschitz condition on \(\Psi\), such functionals are conditionally Wasserstein continuous, as the following result shows.

\begin{proposition}\label{lem_meta_Lipschitz}
Consider the functional \(\kappa\) defined in \eqref{kappa_figalli}. Assume that \(\Psi\) satisfies the Lipschitz condition 
\begin{align}\label{Lipschitz_condition_Psi}
    |\Psi(\pi_2,\pi_x) - \Psi(\tilde \pi_2,\tilde{\pi}_{\tilde{x}})| \leq c\cdot\big[\cW(\pi_x,\tilde{\pi}_{\tilde{x}}) +\cW(\pi_2, \tilde \pi_2)\big] %\quad \text{for }\mu\text{-almost all } x,\tilde{x}\in \cX.
\end{align}
for all \(\pi,\tilde{\pi}\in \mathcal{P}(\cX\times \cY)\) and \(x,\tilde{x}\in \cX\).
Then \(|\kappa(\pi^n)-\kappa(\pi)| \leq 2c \cdot\cIW(\pi^n,\pi)\), in particular \(\cIW(\pi^n,\pi)\to 0\)  implies \(\kappa(\pi^n) \to \kappa(\pi)\).
\end{proposition}

\begin{proof}
    By Lipschitz continuity we have $$|\kappa(\pi^n) - \kappa(\pi)| \leq c \Big[\int \cW(\pi^n_x,\pi_{\tilde{x}})\, \gamma(\de x,\de \tilde{x})+\cW(\pi_2^n, \pi_2)\Big]$$ for all \(n\) and for all couplings \(\gamma\in \Pi(\pi_1^n,\pi_1)\). Taking the infimum over all \(\gamma\) and using Lemma \ref{lem:second} implies \(|\kappa(\pi^n)-\kappa(\pi)| \leq 2c \cdot\cIW(\pi^n,\pi)\).
\end{proof}

A class of dependence measures that intrinsically satisfy the Lipschitz condition \eqref{Lipschitz_condition_Psi} are the optimal transport-based Wasserstein correlations in \cite{wiesel2022}. To be precise, the Wasserstein correlation \(\dW(\pi)\) is defined for \(\pi\in \mathcal{P}(\mathcal{X}\times \mathcal{Y})\) via 
\begin{align*}
    \dW(\pi) := \frac{\int \cW(\pi_2,\pi_x) \, \pi_1(\de x)}{\int d(y,\tilde{y}) \,\pi_2(\de y) \pi_2(\de \tilde{y})}.
\end{align*}
By the triangle inequality for the Wasserstein distance, we have $$|\cW(\pi_2,\pi_x) - \cW(\tilde\pi_2,\tilde{\pi}_{\tilde{x}})| \leq \cW(\pi_x,\tilde{\pi}_{\tilde{x}}) +\mathcal{W}(\pi_2, \tilde\pi_2).$$ Hence, the following result is a direct consequence of Proposition \ref{lem_meta_Lipschitz}. Recall that $\pi_2$ is non-degenerate, which guarantees that the denominator of \(\dW(\pi)\) is strictly positive.

\begin{corollary}
    If \(\cIW(\pi^n,\pi)\to 0\), then \(\dW(\pi^n) \to \dW(\pi).\)
\end{corollary}

As a second example, take $\mathcal{X}=\mathcal{Y}=\mathcal{H}$, where $\mathcal{H}$ is an RKHS with bounded feature map $\phi:\R^d\to \mathcal{H}.$ \emph{Kernel partial correlation} \cite{deb2020b} is defined as
\begin{align*}
    \operatorname{KPC}(\pi) := \frac{\int \mathrm{MMD}^2(\pi_x, \pi_2)\,\pi_1(\de x)}{\int \mathrm{MMD}^2(\delta_y,\pi_2)\,\pi_2(\de y)},
\end{align*}
where 
\begin{align*}
    \mathrm{MMD}(\mu,\nu):= \sup_{\|\gamma\|_{\mathcal{H}} \le 1} \int \langle \gamma, x
    \rangle \,\mu(\de x)- \int \langle \gamma, x\rangle \,\nu(\de x)
\end{align*}
for $\mu,\nu\in \mathcal{P}(\mathcal{H})$. It can be shown that 
\begin{align*}
   \big| \mathrm{MMD}^2(\pi_x, \pi_2) - \mathrm{MMD}^2(\tilde \pi_{\tilde{x}}, \tilde \pi_2) \big| \le c \cdot [\mathcal{W}(\pi_x,\tilde \pi_{\tilde x}) + \mathcal{W}(\pi_2, \tilde \pi_2) \big]; 
\end{align*}
see Lemma \ref{lem: kpc}. We thus conclude the following.

\begin{corollary}
    If \(\cIW(\pi^n,\pi)\to 0\), then \(\mathrm{KPC}(\pi^n) \to \mathrm{KPC}(\pi).\)
\end{corollary}

As we will show later in this article, the other classes of dependence measures mentioned above are also subject to the concept of conditional Wasserstein continuity. Furthermore, we use this concept to establish convergence properties of estimators with respect to several modes of convergence that imply, in particular, weak convergence of conditional distributions.

While the conditional Wasserstein distance applies to dependence measures on general Polish spaces, we will subsequently restrict our attention to the Euclidean setting, where explicit convergence rates can be established.

\subsection{Rates of convergence for copula estimators}

In the remainder of this section, we focus on the univariate case $\mathcal{X} = \mathcal{Y} = [0,1]$. There are different ways to extend the below results to the multivariate setting; see e.g. \cite{Beiglboeck-Pammer-Posch-2023} and the discussion in Remark \ref{rem_ACE_CCE}.
The aim of this section is to introduce the so-called adapted copula estimator and determine its rates of convergence by leveraging recent results on the adapted empirical measure established in~\cite{Backhoff-Bartl-Beiglboeck-Wiesel-2022}. 
Similar techniques allow us to resolve an open problem: deriving rates of convergence for the well-known checkerboard estimator of copulas under metrics that account for weak convergence of conditional distributions \cite{strothmann2022}. Subsequently, we determine convergence rates for plug-in estimators of dependence measures constructed via the adapted copula estimator and the checkerboard copula approaches.

%In the remainder of this section we focus on the univariate case $\mathcal{X}=\mathcal{Y}=[0,1]$ and assume that $\pi$ has uniform marginals $\mathcal{U}$. \textcolor{red}{There are different ways to extend the below results to the multivariate setting, see e.g.~\cite{Beiglboeck-Pammer-Posch-2023}.... We leave this to future research.} \JA{habe die Motivation mehr auf rates of convergence for checkerboard estimator und dependence measures gelenkt. Später können wir noch erwähnen, wie man den adapted copula estimator auf ein allgemeineres Setting erweitern könnte. Passt das?} \textcolor{red}{Passt für mich. Ist der nächste Abschnitt, dann noch okay, oder sollen wir den auch anpassen?} \JA{vielleicht etwas feinschleifen:)}

In the setting $\mathcal{X} = \mathcal{Y} = [0,1]$ it is natural to represent $\pi$ by the copula $C_\pi$, that corresponds to $\pi$ via Sklar's theorem; see Section \ref{sec:prelim} for a primer on copulas and dependence modeling. While we have explained above, how to obtain a suitable distance on $\mathcal{P}(\cX\times \cY)$ that enforces continuity of dependence measures, one can ask the same question for the space of copulas. It turns out that on this space, a suitable distance is given by the $d_1$-metric, which is defined as 
\begin{align}
    d_1(C,\tilde C):=\int_0^1 \int_0^1 \left| \partial_1 C(u,v) - \partial_1 \tilde C(u,v)\right|\, \de u\,\de v;
\end{align}
see \cite[Def.~2.21]{Ansari-Rueschendorf-2021} and also Section \ref{sec:copulas} below. In particular, we show in Lemma \ref{lem:KR} that $$d_1(C_\pi, C_{\tilde\pi})=  \mathcal{KR}(\pi, \tilde \pi)\ge \mathcal{AW}(\pi, \tilde \pi)$$ for all $\pi, \tilde \pi$ with uniform marginals on \([0,1]\), where $\mathcal{KR}$ is the so-called \emph{Knothe-Rosenblatt distance}; see \eqref{eq:KR} in Section \ref{sec:ot}. In Section \ref{sec:prelim} we also discuss other relationships between convergence in $\mathcal{P}(\R\times \R)$ and convergence in the space of copulas.
%
%In contrast to classical measures of association---such as Pearson's correlation, Kendall's tau, or Spearman's rho---dependence measures that quantify the strength of functional dependence rely inherently on conditional distributions. Estimating these quantities therefore requires assessing conditional probabilities; consequently, as mentioned above, classical estimation methods based on weak continuity do not apply. To this end, we focus on stronger modes of convergence. \textcolor{red}{This was already mentioned above - if you are okay with, I would significantly cut this.}

We now turn to estimation of dependence measures. Classically, this issue has been addressed in the literature through copula estimators. In this paper we construct a new copula estimator, show its consistency and estimate its sample complexity. 
%\JA{we should somewhere discuss/mention how the adapted copula estimator behaves outside the \(\cU\)-marginal setting. If this is not clear, it might rather be a tool to derive rates of convergence for the checkerboard copula estimator. Then we should not focus on the new method. Können wir morgen diskutieren}

To detail its construction, we assume that the pairs $(X_l, Y_l)_{l=1}^N$ are i.i.d.~samples from $\pi$ and recall the \emph{adapted empirical measure}
\begin{align*}
    \hat{\pi}^N = \frac{1}{N} \sum_{l=1}^N \delta_{(\varphi^N(X_l), \varphi^N(Y_l))}
\end{align*}
from \cite{Backhoff-Bartl-Beiglboeck-Wiesel-2022}; here $\varphi^N:\R\to \R$ is a binning function with finite range. We refer to Section \ref{sec:estimation} for a detailed description. In order to determine a corresponding copula we consider the \emph{continuous adapted empirical measure} \(\hat{\pi}_c^N\), that is defined by uniformly distributing mass on the bins obtained from $\varphi^N(\R)$; see \eqref{eq:def_cont}. We call the induced copula the \emph{adapted empirical copula} 
\begin{align*}
    \hat{C}_N := C_{\hat{\pi}^N_c},
\end{align*}
where we refer to the explanations surrounding \eqref{def_cop_est} for details. Our second main contribution is the consistency result
\begin{align*}
    \|\hat{C}_N - C_\pi\|_\infty \to 0, \quad d_1(\hat{C}_N ,C_\pi)\to 0, \quad \mathcal{KR}(\hat{\pi}^N_c, \pi) \to 0
\end{align*}
in Theorems \ref{lem:triangle_est} and \ref{thm_conv_C_N}. Further, when $\pi$ has Lipschitz kernels (see Def. \ref{def:lipschitz_kernel}) and its second marginal has Lipschitz continuous cdf, then 
\begin{align*}
\E[\mathcal{CW}(\hat{\pi}_c^N, \pi)]\le \E[\mathcal{AW}(\hat{\pi}_c^N, \pi)]\le  \E[\mathcal{KR}(\hat{\pi}_c^N, \pi)] \le cN^{-1/3}
\end{align*}
and
\begin{align*}
\E[d_1(\hat{C}_N,C_\pi)] \le cN^{-1/3},
\end{align*}
as stated in Theorem \ref{lem:triangle_est} and Corollary \ref{thm_rates_C_N}.

The adapted empirical copula estimator is different from, but related to, the well-known \emph{checkerboard copula estimator} \(\hat{C}_N^\#\) studied in \cite{fgwt2021,strothmann2022}.
The latter is defined as the copula of the empirical, rank-based measure \(\hat{\pi}_c^{N,\#}\), i.e.
\begin{align*}
    \hat{C}_N^\# := C_{\hat{\pi}_c^{N,\#}}.
\end{align*}
Here, \(\hat{\pi}_c^{N,\#}\) is a continuous modification of the adapted empirical measure applied to the \emph{rank-transformed} observations such that \(\hat{\pi}_c^{N,\#}\) has uniform marginals on \((0,1)\); see \eqref{def_emp_checkerboard_measure}.
It is well known that the checkerboard estimator is a strongly consistent estimator of \(C_\pi\). In fact, Lemma \ref{lem_conv_C_N_check} and Corollary \ref{cor:checkerboar} show that it satisfies 
\begin{align*}
     \| \hat{C}_N^\# - C_\pi\|_\infty \to 0, \quad d_1( \hat{C}_N^\# ,C_\pi)\to 0, \quad \mathcal{KR}(\hat{\pi}^{N,\#}_c, \pi) \to 0.
\end{align*}
%and exhibits the same convergence properties as those in Theorem \ref{thm_conv_C_N}
As our third main contribution, we establish convergence rates of the checkerboard estimator $\hat{C}^{\#}_N$ with respect to metrics that capture conditional weak continuity such as \(d_1\), \(\cAW\), and \(\mathcal{KR}\). To the best of our knowledge, these are the first rates available in the literature. In particular, our results resolve an open question of \cite[Section 5]{strothmann2022}. To be precise, according to Theorem \ref{thm:rates} we find that
\begin{align*}
    \E[\mathcal{CW}(\hat{\pi}_c^{N,\#}, \pi)]\le \E[\mathcal{AW}(\hat{\pi}_c^{N,\#}, \pi)]\le \E[\mathcal{KR}(\hat{\pi}_c^{N,\#}, \pi)]\le cN^{-1/3}
\end{align*}
and
\begin{align*}
    \E[d_1(\hat{C}_N^{\#},C_\pi)] \le cN^{-1/3}
\end{align*}
under a Lipschitz assumption on the kernels of $\pi$.

\subsection{Rates of convergence for dependence measures}

As our fourth main contribution, we determine rates of convergence for plug-in estimators of rank-based and rearranged dependence measures. 
We adopt the ideas of \cite[Section 6]{wiesel2022} where rates of convergence have been established for Wasserstein correlations. Here, we consider rank-based dependence measures of the form 
\begin{align}
    \xi_\varphi(\pi) = \xi_\varphi(C_\pi)= \frac{\int  \varphi(\P(Y\geq y \mid X=x) - \P(Y\geq y)) \,\pi_2(\de y) \pi_1(\de x)}{\int  \varphi(\1_{\{\tilde{y} \leq y\}} - \P(Y\geq y))\, \pi_2(\de y) \pi_2(\de \tilde{y})} 
\end{align}
for \((X,Y)\sim \pi\) with uniform marginals $\mathcal{U}$. We assume that \(\varphi\) is convex and strictly convex at \(0 = \varphi(0)\) so that \(\xi_\varphi\) satisfies Axioms \ref{axiom1}--\ref{axiom3}; see \cite{Ansari-Fuchs-2026}. Specifically, for \(\varphi(x) = x^2\), the functional \(\xi_\varphi\) coincides with Chatterjee's rank correlation \cite{chatterjee2020}, whose nearest neighbour-based estimator achieves optimal convergence rates of order \(N^{-1/2}\) \cite{chatterjee2021,han2022limit}. 
In the general case, we obtain from Theorems \ref{thm_rates_C_N} and \ref{thm:rates} rates of order \(N^{-1/3}\) for plug-in estimators of \(\xi_\varphi\) which are based on the adapted empirical copula estimator \(\hat{C}_N\) and the checkerboard copula estimator \(\hat{C}_N^\#\) as stated in Corollary \ref{cor:rates}. In particular, assuming that \(\pi\in \Pi(\cU,\cU)\) has Lipschitz kernels, we find
\begin{align*}
    \E[|\xi_\varphi(\hat{C}_N) - \xi_\varphi(C_\pi)|] \le cN^{-1/3} \quad \text{and} \quad \E[|\xi_\varphi(\hat{C}_N^\#) - \xi_\varphi(C_\pi)|] \le cN^{-1/3}.
\end{align*}
For a similar result on rates of convergence for measures of sensitivity, we refer to Section \ref{sec_rank_based_DM}. Another class of dependence measures recently studied in \cite{strothmann2022} are the so-called rearranged dependence measures. For a functional \(\eta\) that satisfies Axioms \ref{axiom1}--\ref{axiom3} on the subclass of stochastically increasing copulas, the rearranged dependence measure \(\sR_\eta\) is defined by
\begin{align}\label{def_RDM0}
    \sR_\eta(\pi) := \eta(C_\pi^\uparrow).
\end{align}
Here, \(C_\pi^\uparrow\) is the copula of the increasing rearrangement of \(\pi\); see Definition \ref{lem:delta_1}. An important example for \(\eta\) is Spearman's rank correlation \(\varrho\). Recall that \(\varrho\) admits a representation through the underlying copula \(C_\pi\) by
\begin{align*}
    \varrho(C_\pi) = 12 \int_0^1 \int_0^1 C_\pi(u,v)\, \de u \de v - 3;
\end{align*}
see e.g.~\cite[Theorem 5.1.6]{Nelsen-2006}. Plugging in the adapted or the checkerboard copula estimator, we obtain the following rates of convergence for the estimators \(\hat{\sR}_\varrho^N = \rho(\hat{C}_N^\uparrow)\) and \(\hat{\sR}_\varrho^{N,\#} = \rho((\hat{C}_N^\#)^\uparrow)\) in Corollary \ref{cor_Spearman}: we find that, if \(\pi\) has Lipschitz kernels, then 
\begin{align*}
    \E[|\hat{\sR}^N_\varrho - \sR_\varrho(\pi)|] \le cN^{-1/3} \quad \text{and} \quad \E[|\hat{\sR}^{N,\#}_\varrho - \sR_\varrho(\pi)|] \le cN^{-1/3}.
\end{align*}    
This also resolves an open problem in \cite[Section 5]{strothmann2022}.
More generally, we determine rates of convergence for estimators of \(\sR_\eta\) under a Lipschitz condition on \(\eta\) similar to \eqref{Lipschitz_condition_Psi} in Theorem \ref{cor:rates_rear}.
%Then, using Theorem \ref{thm:rates}, we obtain the same rates of convergence when plugging in the checkerboard estimator into the rearranged dependence measure \eqref{def_RDM0}.

\subsection{Literature review}
\textbf{Optimal transport:}~We refer to \cite{Ru-1998, Villani-2009, filippo, ambrosio2005gradient} for an overview of optimal transport theory. Adapted optimal transport has been introduced in a series of papers \cite{Al81, hoover1984adapted, ruschendorf1985wasserstein, hellwig1996sequential, lassalle2018causal, acciaio2019extended, bonnier2023adapted}. We will particularly make use of results from \cite{Beiglboeck-Pammer-Posch-2023} in the discussion below. Interestingly, \cite{backhoff2020all} states that all adapted topologies introduced in these papers coincide. \\
\textbf{Dependence modeling: } We refer to~\cite{fdsempi2016,Nelsen-2006} for a comprehensive overview of 
copula theory. Various metrics and modes of convergence for copulas have 
been investigated in~\cite{Ansari-Rueschendorf-2021,darsow-1992,Li-Mikusinski-Taylor-1998,Mikusinski-Taylor-2009,Mikusinski-2010,wt2011}; 
see~\cite[Section~4]{fdsempi2016} for an overview. Weak convergence of the empirical copula process is studied in \cite{Chen-Huang-2007,Fermanian-Radulovic-Wegkamp-2004,Genest-Neslehova-Remillard-2017,Omelka-Gijbels-Veraverbeke-2009,Rueschendorf-1976}, strong consistency of the checkerboard copula estimator with respect to the \(d_1\)-metric is documented in \cite{fgwt2021,strothmann2022}.\\
\textbf{Dependence measures:} As mentioned above, in recent years there have been many proposals for dependence measures satisfying Axioms \ref{axiom1}--\ref{axiom3}; see e.g.~\cite{chatterjee2021,Azadkia-2025,Figalli-2025,chatterjee2020,deb2020,bernoulli2021,deb2020b,hoermann2025,strothmann2022,wiesel2022} and the references therein. \\
%\JA{Literature: Standard extension of copulas (linearly interpolated) are given in \cite{darsow-1992}, \cite{Neslehova-2007}}
The connections between dependence modeling and adapted optimal transport are relatively unexplored except for \cite{wiesel2022} to the best of our knowledge. This paper aspires to fill this gap. We hope to provide a bridge between the fields of OT, statistics and copula theory, clarifying their various interconnections.

\subsection{Structure of the paper}

The rest of the paper is organized as follows.
Section \ref{sec:prelim} introduces the notation and the basic concepts of optimal transport and dependence modeling that we use in this paper.
Section \ref{sec:comparison} provides a comparison between different notions of convergence arising in optimal transport and copula theory. In particular, we show that convergence in adapted Wasserstein distance and \(d_1\)-convergence of copulas are equivalent; see Corollary \ref{cor:d1}. For stochastically increasing distributions, these modes of convergence are also equivalent to metrics of weak convergence; see Corollary \ref{cor:si}.
In Section \ref{sec:estimation}, we focus on copula estimators derived from the adapted empirical measure and the checkerboard approach. While the checkerboard estimator---frequently used in the dependence modeling community---relies on binning rank-transformed observations, the adapted empirical copula operates directly on the original observations. Yet, as established in Corollary \ref{thm_rates_C_N} and Theorem \ref{thm:rates}, both estimators achieve the same rates of convergence---with respect to metrics of conditional weak convergence. 
In Section \ref{sec_rank_based_DM}, we provide strongly consistent plug-in estimators for rank-based dependence measures and determine their speed of convergence. 
Estimating rearranged dependence measures is covered in Section \ref{sec_RDM}.
Some proofs and auxiliary results are deferred to the appendix.

\begin{comment}
\begin{example}
    For \(\cX = \R^d\), \(\cY = \R\), and \((X,Y)\sim \pi\in \mathrm{Prob}(\R^p,\R)\) with \(\pi_2 = \nu\) non-degenerate, define for \(p\geq 1\),
    \begin{align}
        \xi_p(\pi) := \xi_p(X,Y) := \frac{\int_{\R^p} \int_\R |P(Y\geq y\mid X = x) - P(Y\geq y) | \de \nu(y) \de \mu(x)}{\int_{\R^p} \int_\R |\1_{\{Y\geq y\}} - P(Y\geq y) | \de \nu(y) \de \mu(x)}.
    \end{align}
    Then \(\xi_p\) is a dependence measure that satisfies Axioms \ref{axiom1}--\ref{axiom3}; see ... For \(p=2\), it reduces to the population version of Chatterjee's rank correlation \cite{chatterjee2021}. The inner integral \(\Psi(\nu,\pi_x) = \int_\R |P(Y\geq y\mid X=x) - P(Y\geq y)|^p \de \nu(y)\) satisfies the Lipschitz condition \eqref{Lipschitz_condition_Psi} since
    \begin{align*}
         & \int_\R |P(Y\geq y\mid X=x) - P(Y\geq y)|^p \de \nu(y)  - \int_\R |P(\tilde{Y}\geq y\mid \tilde{X}=\tilde{x}) - P(\tilde{Y}\geq y)|^p \de \nu(y) \\
         &\geq p\, \int_\R |P(Y\geq y\mid X=x) - P(\tilde{Y}\geq y\mid \tilde{X}=\tilde{x})| \de \nu(y) = p\, \cW(\pi_x, \tilde{\pi}_{\tilde{x}}),
    \end{align*}
    where \((\tilde{X},\tilde{Y})\sim \tilde{\pi}\).
    \JA{do we need continuous distribution functions for the equality?}
    Hence, \(\pi^N \to \pi\) in \(\cIW\) implies \(\xi_p(\pi^N) \to \xi_p(\pi)\) by Lemma \ref{lem_meta_Lipschitz}.
\end{example}
\end{comment}

\section{Notation and preliminary results} \label{sec:prelim}

\subsection{General notation}
Let us fix a probability space $(\Omega, \F, \P)$ on which all random elements are defined. Throughout the paper, we consider Polish spaces $\mathcal{X}$ with a compatible metric $d$; in particular we equip the space $\R^{k}$, $k\ge 1$,  with the Euclidean norm $|\cdot|$. We write $\mathcal{P}(\mathcal{X})$ for the set of (Borel) probability measures on $\mathcal{X}$ and write $X\sim \pi$ for a random variable $X$, if $X$ has law $\pi\in \mathcal{P}(\mathcal{X})$. For random variables $X$ and $Y$, we write $X\eqd Y$ whenever they have the same distribution. We write $F_X$ for the cdf of $X$, and if $\mathcal{X}=\R$, we write $F_X^{-1}$ for the left-continuous quantile function of $X$. Similarly we write $F_\pi$ for the cdf of $\pi$. We write $\mathcal{U}^A$ for the uniform distribution on a Borel set $A\subseteq \mathcal{X}$ and define $\mathcal{U}:=\mathcal{U}^{[0,1]}$, i.e.~the uniform distribution on the unit interval $[0,1]\subset \R.$ For two measures $\pi, \tilde\pi \in \mathcal{P}(\mathcal{X})$ we denote by $(\pi\otimes \tilde \pi)(A\times B):=\pi(A)\tilde \pi(B)$ for Borel $A,B\subseteq \mathcal{X}$ the product coupling of $\pi, \tilde\pi$. 
For a Borel function $f:\mathcal{X}\to \mathcal{Y}$ we write $f_\#\pi$ for the pushforward measure of $\pi$ through $f,$ i.e. 
$$(f_\#\pi)(A):=\pi(\{x\in \mathcal{X}:\ f(x)\in A\})$$
for all Borel sets $A\subseteq \mathcal{Y}$.  
For a pair $(X,Y)\sim \pi\in \mathcal{P}(\mathcal{X}\times \mathcal{Y})$ we denote the marginal laws by $\pi_1(\cdot)= \P(X\in \cdot)$, $\pi_2(\cdot)= \P(Y\in \cdot)$ and recall the disintegration rule
\begin{align} \label{eq:disintegration}
 \pi(A\times B)=\P((X,Y) \in A\times B)=\int_A \P(Y\in B| X=x) \,\pi_1(\de x) =\int_A \pi_x(B)\,\pi_1(\de x)=: (\pi_1\otimes \pi_x)(A\times B)    
\end{align}
for all Borel sets $A\subseteq \mathcal{X}, B\subseteq \mathcal{Y}$, where $\mathcal{X}\ni x\mapsto \P(Y\in\cdot|X=x) =:\pi_x(\cdot)\in \mathcal{P}(\mathcal{Y})$ is a Borel measurable function. Similarly to above, for $\pi\in \mathcal{P}(\mathcal{X}\times \R)$ and fixed $x\in \mathcal{X}$, we write $F_{Y|X=x}$ and $F^{-1}_{Y|X=x}$ for the cdf and quantile function of the distribution $\P(Y\in\cdot|X=x).$

We make the convention that we write $x,y$ when considering vectors in $\mathcal{X}$ and $\mathcal{Y}$ respectively, and $u,v$ when considering numbers in $[0,1]$. Throughout we denote generic constants (which might increase from line to line) by $c>0.$

\subsection{Optimal transport}\label{sec:ot}

For the remainder of the paper we set $\mathcal{X}=\R^k$ and $\mathcal{Y}=\R^l$, $k,l\in \N.$
In this section we define some basic distances in Optimal Transport (OT) theory. We refer to \cite{Ru-1998, Villani-2003,Villani-2009} for details.

We define the set of couplings of two probability measures $\pi,\tilde \pi\in \mathcal{P}(\R^{d})$, $d\ge 1$, via
\begin{align*}
\Pi(\pi,\tilde \pi) := \left\{\gamma \in \mathcal{P}(\R^{d}\times \R^{d}):\, \ \gamma(A\times \R^{d})=\pi(A), \ \gamma(\R^{d}\times A)=\tilde \pi(A)\quad \forall \text{ Borel }A\subseteq \R^d\right\}.
\end{align*}
The \emph{Wasserstein distance} between two laws $\pi, \tilde \pi \in \mathcal{P}(\R^d)$ is defined as
\begin{align*}
\mathcal{W}(\pi,\tilde \pi)=\inf_{\gamma \in \Pi(\pi,\tilde \pi)} \int |x-y|\,\gamma(\de x,\de y).
\end{align*}
To simplify notation below, we often consider pairs of random vectors $(X,Y)\sim \pi \in \mathcal{P}(\R^{k}\times \R^l)$ and $(\tilde X,\tilde Y)\sim \tilde \pi\in \mathcal{P}(\R^{k}\times \R^l)$.
As mentioned in the Introduction, we will consider several refinements of $\mathcal{W}$ in this paper. The first of these is the \emph{adapted Wasserstein distance}, which was defined in \eqref{IW_leq_AW} as
\begin{align*}
    \mathcal{AW}(\pi, \tilde \pi) = \inf_{\gamma \in \Pi(\pi_1, \tilde \pi_1)} \int |x-\tilde x|+\mathcal{W}(\pi_x, \tilde \pi_{\tilde x})\,\gamma(\de x, \de \tilde x) 
\end{align*}
for $\pi, \tilde \pi\in \mathcal{P}(\R^{k}\times \R^l).$ For $l=1$ we can explicitly compute the  term $\mathcal{W}(\pi_x, \tilde \pi_{\tilde x})$ and find the representation 
\begin{align} \label{eq:AW_simple}
\begin{split}
    \mathcal{AW}(\pi, \tilde \pi) &= \inf_{\gamma \in \Pi(\pi_1, \tilde \pi_1)} \int |x-\tilde x|+\int_0^1 |F_{Y|X=x}^{-1}(v)-F_{\widetilde Y|\widetilde X= \tilde x}^{-1}(v)|\,\de v \,\gamma(\de x, \de \tilde x)\\
    &= \inf_{\gamma \in \Pi(\pi_1, \tilde \pi_1)} \int |x-\tilde x|+\int |F_{Y|X=x}(y)-F_{\widetilde Y|\widetilde X= \tilde x}(y)|\,\de y \,\gamma(\de x, \de \tilde x)
\end{split}    
\end{align}
where we recall $(X,Y)\sim \pi, (\tilde X, \tilde Y) \sim \tilde \pi$; we refer to Lemma \ref{lem:L1_integral} for the second equality. We will use the representation \eqref{eq:AW_simple} frequently. 

Next we consider the \emph{Knothe-Rosenblatt distance} $\mathcal{KR}$. For notational simplicity we only give its definition for $k=l=1$ and refer to \cite{Beiglboeck-Pammer-Posch-2023} for a more general definition. For this, we note that choosing $\gamma \in \Pi(\pi_1, \tilde \pi_1)$ in \eqref{eq:AW_simple} equal to the quantile coupling $(F_{X}^{-1}, F_Y^{-1}) _{\#}\mathcal{U}$  gives rise to
\begin{align}\label{eq:KR}
    \mathcal{KR}(\pi, \tilde \pi) := \int_0^1 |F_X^{-1}(u) - F_{\widetilde X}^{-1}(u)| + \int_0^1 |F^{-1}_{Y|X=F_{X}^{-1}(u)}(v)-F^{-1}_{\widetilde Y|\widetilde X= F_{\widetilde X}^{-1}(u)}(v)|\,\de v \,\de u.
\end{align}
%
%Similarly we call 
%\begin{align*}
%\Big(\Big(F_{X}^{-1}(u),F^{-1}_{Y|X=F_{X}^{-1}(u)}(v) \Big) , \Big(F_{\widetilde{X}}^{-1}(u), F^{-1}_{\widetilde Y|\widetilde X= F_{\widetilde X}^{-1}(u)}(v)\Big) _{\#} (\mathcal{U}\otimes \mathcal{U})(\de u, \de v)   
%\end{align*}
%the \emph{Knothe-Rosenblatt rearrangement}; see \cite{knothe1957contributions} \textcolor{red}{I think this is actually never used.}. 
The Knothe-Rosenblatt distance $\mathcal{KR}$ is a metric on the space of probability measures; it is investigated in detail in \cite{Beiglboeck-Pammer-Posch-2023}. From \eqref{eq:AW_simple} and \eqref{eq:KR} we clearly have $\mathcal{AW}\le \mathcal{KR}$, where the inequality is strict in general. However, the topologies generated by $\mathcal{AW}$ and $\mathcal{KR}$ are equal in the following sense.
\begin{lemma}[{\cite[Theorem 1.4]{Beiglboeck-Pammer-Posch-2023}}] \label{lem:beiglboeck-pammer}
    For a sequence $(\pi^n)$ in $\mathcal{P}(\R^2)$ and $\pi\in \mathcal{P}(\R^2)$ we have
    \begin{align*}
        \mathcal{AW}(\pi^n, \pi)\to 0 \quad \Leftrightarrow \quad  \mathcal{KR}(\pi^n, \pi) \to 0.
    \end{align*}
\end{lemma}
In this paper we focus on a third distance as introduced in \eqref{eq:causal}, namely the \emph{conditional Wasserstein distance}. For the choice $\mathcal X=\R^k, \mathcal{Y}=\R^l$ it is defined as
\begin{align*}
    \mathcal{CW} (\pi,\tilde \pi) = \inf_{\gamma\in \Pi(\pi_1, \tilde{\pi}_1)} \int \mathcal{W}(\pi_x,\tilde \pi_{\tilde x})\,\gamma(\de x,\de \tilde x) 
\end{align*}
where $\pi, \tilde \pi \in \mathcal{P}(\R^{k}\times \R^l)$. Note that the conditional Wasserstein distance $\mathcal{CW}$ is a pseudo-metric on $\mathcal{P}(\R^{k}\times \R^l)$; see Example \ref{ex:pseudo} below. In fact it is a Wasserstein metric on $\mathcal{P}(\mathcal{P}(\R^{l}), \mathcal{W}_{\mathcal{P}(\R^{l})})$: recalling the disintegration formula $\pi =  \pi_1\otimes \pi_x$ as stated in \eqref{eq:disintegration}, we can consider the law $(x\mapsto \pi_{x})_{\#} \pi_1\in \mathcal{P}(\mathcal{P}(\R^{l}))$. Then the identity
\begin{align*}
    \mathcal{CW} (\pi,\tilde \pi)= \mathcal{W}_{\mathcal{P}(\R^l)}((x\mapsto \pi_{x})_{\#} \pi_1, (x\mapsto \tilde \pi_{x})_{\#} \tilde \pi_1)
\end{align*}
holds. 

\begin{example}\label{ex:pseudo}
Take $\pi=(u, 1-u)_{\#}\mathcal{U}\in \Pi(\mathcal{U}, \mathcal{U})$ and $\tilde \pi=(u,u)_{\#}\mathcal{U}$. Note that $\pi \neq \tilde \pi$. However, choosing $\gamma=(1-u,u)_{\#}\mathcal{U}$ we find
\begin{align*}
\mathcal{CW}(\pi, \tilde \pi)&=\inf_{\gamma\in \Pi(\cU,\cU)} \int \cW(\pi_u,\tilde \pi_{\tilde{u}})\, \gamma(\de u,\de \tilde{u}) \le \int \cW(\delta_u,\delta_{u})\, \de x  = 0.
\end{align*}
\end{example}

The distance $\mathcal{CW}$ is new to the best of our knowledge. It can be recovered from $\mathcal{AW}$ by setting the distance between the coordinates $(x,\tilde x)$ equal to zero. In spirit, it is also related to \cite[Proposition 12.4.6]{ambrosio2005gradient}, who consider 
\begin{align}\label{eq:gigli}
    \int \mathcal{W}(\pi_x,\tilde\pi_x)\,\pi_1(\de x)
\end{align}
for measures $\pi, \tilde\pi \in \mathcal{P}(\R^{k}\times \R^l)$ with same first marginal $\pi_1=\tilde \pi_1$. Choosing the identity coupling for the first marginals, \eqref{eq:gigli} is clearly an upper bound of $\mathcal{CW}(\pi, \tilde \pi).$
We also record the following property of $\mathcal{CW}:$
\begin{lemma}\label{lem:second}
    For $\pi, \tilde \pi\in \mathcal{P}(\R^{k}\times \R^l)$ we have
    \begin{align*}
        \mathcal{W}(\pi_2, \tilde\pi_2) \le \mathcal{CW}(\pi, \tilde\pi).
    \end{align*}
\end{lemma}
\begin{proof}
    For any $\gamma\in \Pi(\pi_1, \tilde\pi_1)$ we conclude from convexity of the Wasserstein distance (see \cite[Theorem 4.8]{Villani-2009}) that 
    \begin{align*}
         \mathcal{W}(\pi_2, \tilde\pi_2)= \mathcal{W} \Big( \int \pi_x\, \gamma(\de x, \de \tilde x), \int \tilde\pi_{\tilde x} \,\gamma(\de x,\de \tilde x)\Big)\le \int \mathcal{W}(\pi_x, \tilde \pi_{\tilde x})\,\gamma(\de x, \de \tilde x).
    \end{align*}
    Optimizing over $\gamma$ yields the claim.
\end{proof}

We summarize the relationships between the different distances introduced in this section as follows: we have
\begin{align} \label{eq: general_inequalities}
\mathcal{W}(\cdot\,_2,\cdot\,_2)\le
    \begin{matrix}
    \mathcal{W}\\
     \mathcal{CW} 
     \end{matrix} 
     \; \le \mathcal{AW} \le \mathcal{KR}.
\end{align}

\subsection{Dependence modeling and copulas} \label{sec:copulas}

In this section we recall some concepts of dependence modeling and copula theory. We refer to \cite{fdsempi2016, Nelsen-2006} for details. 

A \emph{copula} $C:[0,1]^2 \to [0,1]$ is a distribution function on \([0,1]^2\) with uniform marginals (i.e.~\(C(u,1) = u = C(1,u)\) for all \(u\in [0,1]\)).
Translating to OT language, a copula $C$ is the distribution function of a coupling \(\pi=\pi_C \in \Pi(\cU,\cU)\), where we recall that \(\cU\) denotes the uniform distribution on \([0,1]\). We call $\pi_C$ the \emph{coupling induced by $C$}.
Due to Sklar's theorem, there is a one-to-one correspondence between the probability measures
\begin{align*}
    \mathcal{P}_c(\R^2):= \{\pi \in \mathcal{P}(\R^2): (x,y)\mapsto F_\pi(x,y) \text{ is continuous}\}
\end{align*}
and the set of copulas together with the marginal cdfs of $\pi$: every $\pi \in \mathcal{P}_c(\R^2)$ can be uniquely decomposed into its marginal distribution functions \(F_X\) and \(F_Y\) (recalling $(X,Y)\sim \pi$) and a copula \(C\), so that
\begin{align}\label{eq:sklar}
    F_{X,Y}(x,y) = C(F_X(x),F_Y(y)) \quad \text{for all } (x,y)\in \R^2;
\end{align}
see \cite{Nelsen-2006}. To make the connection clear, we write $C=C_\pi$ and call $C_\pi$ the \emph{copula induced by $\pi$.}

The most important copulas are the \emph{independence copula} \(C^\perp\) and the \emph{comonotonicity copula} \(C^+\) defined by
\begin{align}
    C^\perp(u,v):= uv \qquad \text{and} \qquad C^+(u,v) := \min\{u,v\}\qquad \text{for all }(u,v)\in [0,1]^2.
\end{align}
These model independence and comonotonicity (i.e. perfect \emph{positive} dependence), respectively. Other important copulas are \emph{checkerboard copulas}. These are copulas having mass uniformly distributed on each cube 
\begin{align}\label{eq:square}
S_{ij}^n:=\Big(\frac{i-1}{n},\frac i n\Big]\times \Big(\frac{j-1}{n},\frac j n\Big] \quad \text{ for } i,j\in \{1,\ldots,n\},
\end{align}
where $n\in \N$. In other words, a checkerboard copula \(C\) admits a Lebesgue density \(c\) such that, for all \(i,j\in \{1,\ldots,n\}\), we have \(c(u,v) = c(\tilde{u},\tilde{v})\) for all \((u,v),(\tilde{u},\tilde{v})\in S_{ij}^n\); see \cite[Section 4.1.1]{fdsempi2016}. Since copulas have uniform marginal distributions, any \(n\)-checkerboard copula corresponds to a doubly stochastic \((n\times n)\)-matrix \(A_n = (a_{ij})_{i,j=1}^n\), i.e.~ \(a_{ij}\geq 0\) and \(\sum_{k=1}^n a_{kj} = \sum_{k=1}^n a_{ik} = 1\) for all \(i,j\in \{1,\ldots,n\}\), where the associated \(n\)-checkerboard copula is given by
\begin{align}\label{def:check_cop}
    C_n^\#[A_n](u,v) &:= \int_0^u\int_0^v c(\tilde{u},\tilde{v}) \,\de \tilde{u}\,\de \tilde{v} \quad \text{and }\quad  c(\tilde{u},\tilde{v}) := n\, a_{ij} \quad \text{ for } (\tilde{u},\tilde{v})\in S_{ij}^n.
\end{align}
For a bivariate copula \(C\), consider the doubly stochastic matrix \(A_n^C = (a_{ij}^C)_{1\leq i,j \leq n}\) with 
\begin{align}\label{eq:CS_ij}
a_{ij}^C =n \,V_C(S_{ij}^n):= n \Big[C\Big(\frac i n, \frac j n\Big) - C\Big(\frac{i-1}{n},\frac j n\Big) - C\Big(\frac i n,\frac {j-1} n\Big) + C\Big(\frac{i-1}{n}, \frac{j-1}{n}\Big)\Big],    
\end{align}
where \(V_C(S_{ij}^n)\geq 0\) is also called the the \emph{\(C\)-volume} of the box \(S_{ij}^n\). Note that \(A_n^C\) is doubly stochastic since \(C\) has uniform marginals. Then the \emph{\(n\)-checkerboard approximation} of \(C\) is defined by
\begin{align}
    C_n^\#[C] := C_n^\#[A_n^C], \quad n\in \N.
\end{align}
Since copulas fully model the dependence structure of a distribution, they also specify the associated conditional distributions. For \(\pi\in \cP_c(\R^2)\), it is well known that
\begin{align}\label{eq:d1}
\begin{split}
    \pi_x((-\infty,y]) &= \lim_{\varepsilon\downarrow 0} \frac{\pi((-\infty,x+\varepsilon]\times (-\infty,y]) - \pi((-\infty,x]\times (-\infty,y])}{\pi_1((-\infty,x+\varepsilon]) - \pi_1((-\infty,x])} \\
    &= \lim_{\varepsilon\downarrow 0} \frac{C_\pi(F_X(x+\varepsilon),F_Y(y)) - C_\pi(F_X(x),F_Y(y))}{F_X(x+\varepsilon) - F_X(x)} \\
    &= \partial_1 C_\pi(F_X(x),F_Y(y)),
\end{split}
\end{align}
for \(\pi_1\)-almost all \(x\) and for all \(y\) outside a \(\pi_2\)-null set that may depend on \(x\); see \cite[Theorem 2.2]{Ansari-Rueschendorf-2021} for the general case, including discontinuous marginal distribution functions.
Here \(\partial_1 C(u,v) := \frac{\de}{\de u} C(u,v)\) denotes the first partial derivative of \(C\).
If \(\pi\in \Pi(\cU,\cU)\), then the above expression simplifies to 
\begin{align}\label{rep_copula_derivative}
    \pi_u([0,v]) = \partial_1 C_\pi(u,v) \quad \text{for }(\mathcal{U}\otimes\mathcal{U})\text{-almost all } (u,v)\in [0,1]^2.
\end{align}

Motivated by the insights above, we focus on the following metric in this note. 

\begin{definition}
For copulas $C, \tilde C$ we define
\begin{align}\label{def_partial1_convergence}
    d_1(C,\tilde C):=\int_0^1 \int_0^1 \left| \partial_1 C(u,v) - \partial_1 \tilde C(u,v)\right|\, \de u\,\de v;
\end{align}
see \cite[Def.~2.21]{Ansari-Rueschendorf-2021}. If $d_1(C_n, C)\to 0$ for a sequence of copulas $(C_n),$ then we say that $(C_n)$ is $\partial_1$-convergent to $C.$
%\textcolor{red}{Maybe also cite origins?} \JA{habe nochmal nachgeschaut: In der Literatur davor wird die stärkere \(\partial\)-convergence betrachtet, definiert über \(\partial_1\)- und \(\partial_2\)-convergence. \cite{wt2011} schaut sich dieselbe Konvergenz in \eqref{def_partial1_convergence} an, aber über Markov-Kerne definiert (also für rechtsseitig stetige Versionen von \(u\mapsto\partial_1 C_n(u,v)\) und \(u\mapsto \partial_1 C(u,v)\)) an. Wolfgang besteht auf der Schreibeweise mit Markov-Kernen (spielt im Integral letztlich keine Rolle). Paper mit der Definition über \eqref{def_partial1_convergence} würde er sofort ablehnen als Reviewer:) Deshalb entscheide gerne du, ob wir \cite{wt2011} (auch wenn er uns nicht reviewen wird) als Referenz angeben und seine Definition zu seinem Unwollen abändern oder auf die Definition in \cite{Ansari-Rueschendorf-2021} verweisen.}
\end{definition}

The distance $d_1$ is a metric on the set of copulas; see \cite{wt2011}. Moreover, for two copulas $C,\tilde C$ we have
\begin{align}\label{eq:l1_d1}
\begin{split}
    \|C-\tilde C\|_1 &:= \int_{[0,1]^2} |C(u,v)-\tilde C(u,v)|\,\de u \,\de v \\
    &\phantom{:}= \int_{[0,1]^2} \Big|\int_0^u \partial_1 C(s,v)-\partial_1 \tilde C(s,v)\,\de s\Big|\,\de u \,\de v\\
    &\phantom{:}\le  \int_{[0,1]^2}\int_0^1 |\partial_1 C(s,v)-\partial_1 \tilde C(s,v)| \de s\,\de u \,\de v = d_1(C, \tilde C).
\end{split}
\end{align}
Hence, by continuity of copulas, \(\partial_1\)-convergence implies pointwise convergence of copulas. The latter is equivalent to uniform convergence because copulas are uniformly Lipschitz in the sense that \(|C(u,v) - C(\tilde u , \tilde v)| \leq |u-v| + |\tilde u - \tilde v|\) for all \(u,v,\tilde u, \tilde v\in [0,1]\) and \(C\in \cC\); see \cite[Theorem 2.2.4]{Nelsen-2006}. This gives
\begin{align*}
    d_1(C_n,C) \to 0 \quad &\Rightarrow \quad C_n(u,v)\to C(u,v) \quad \mathcal{U}([0,1]^2)\text{-a.s.}\quad\\ &\Leftrightarrow \quad \lVert C_n-C\rVert_{\infty} := \sup_{u,v\in [0,1]}|C_n(u,v) - C(u,v)| \to 0
\end{align*}
The following \(d_1\)-approximation through checkerboard copulas is a consequence of \cite[Theorem 5]{Mikusinski-2010}.
\begin{lemma}
    For any \(C\in \cC\), it is \(d_1(C_n^\#[C],C) \to 0\) as \(n\to \infty\).
\end{lemma}
Similar approximations of copulas with respect to \(\partial_1\)-convergence can be obtained by so-called Bernstein and checkmin copulas; see \cite{Mikusinski-2010}. 
While Bernstein copula approximations are based on polynomial approximations, checkmin copulas are a variant of checkerboard copulas where mass is uniformly distributed over the diagonal of each subcube of the checkerboard. Since checkerboard copulas admit a Lebesgue density that is constant on each subcube, we will work for convenience with the checkerboard approximation in \eqref{def:check_cop}.

\subsection{Increasing distributions and rearrangements}

Both in OT and dependence modeling, a prominent role is played by stochastic orders; see \cite{Muller-2002} for an overview. For this article, the following positive/negative dependence concept will be of fundamental importance.

\begin{definition}[Stochastically increasing distributions]\mbox{}
\begin{enumerate}[label = (\roman*)]
    \item A law $\pi \in \mathcal{P}(\R^2)$ is called \emph{stochastically increasing} (SI) if
    \begin{align*}
        x\mapsto \pi_x([y, \infty)) \quad \text{ is increasing for all }y\in \R.
    \end{align*}
    \item A law $\pi\in \mathcal{P}(\R^2)$ is called \emph{stochastically decreasing} (SD), if $(-x,y)_{\#}\pi$ is SI.  
    \item A copula $C$ is SI/SD, if the coupling $\pi_C\in \Pi(\mathcal{U}, \mathcal{U})$ induced by $C$ is SI/SD.
\end{enumerate}
\end{definition}

In other words, if $\pi$ is SI, then $\pi_x$ (first-order) stochastically dominates $\pi_{\tilde x}$ for $x\ge \tilde x$. We denote by
\begin{align*}
    \mathcal{P}^\uparrow(\R^2)&:= \{ \pi \in \mathcal{P}(\R^2): \pi \text{ is stochastically increasing} \},\\
    \mathcal{C}^\uparrow&:= \{C \text{ is stochastically increasing} \}
\end{align*}
the set of SI laws and copulas, respectively. We also note that a copula is SI (SD) if and only if it is concave (convex) in its first component. This follows from the fact that \(u\mapsto \partial_1 C_\pi(u,v) = (\pi_C)_u([0,v]) = 1- (\pi_C)_u((v,1])\) is decreasing (increasing) for all \(v\in [0,1]\) whenever \(C\) is SI.

The importance of SI couplings for OT is summarized in the following lemma:
\begin{lemma}[{\cite[Corollary 2]{ruschendorf1985wasserstein}}]\label{lem_rep_IW_AW}
    Assume that \((X,Y)\sim \pi\) and \((\widetilde{X},\widetilde{Y})\sim \tilde{\pi}\) are both SI or SD. Then the Knothe-Rosenblatt coupling \eqref{eq:KR} is optimal for $\mathcal{CW}$ and $\mathcal{AW}$, that is, we have
    \begin{align}\label{eq:bw_si}
        \mathcal{CW}(\pi,\tilde{\pi}) = \int_{[0,1]^2} \big| F_{Y|X=F_{X}^{-1}(u)}^{-1}(v) - F_{\widetilde{Y}|\widetilde{X}=F_{\widetilde{X}}^{-1}(u)}^{-1}(v)\big| \, \de v\,\de u,
    \end{align}
    as well as 
    \begin{align}\label{eq:aw_si}
        \cAW(\pi,\tilde{\pi}) = \int_0^1 \big|F_{X}^{-1}(u) - F_{\tilde{X}}^{-1}(u)\big| + \int_0^1 \big| F_{Y|X=F_{X}^{-1}(u)}^{-1}(v) - F_{\tilde{Y}|\tilde{X}=F_{\widetilde{X}}^{-1}(u)}^{-1}(v)\big|\,\de v\, \de u .
    \end{align}
\end{lemma}

On the flip side, the importance of SI couplings for dependence modeling is highlighted by the following result: 
\begin{lemma}[{\cite[Proposition 3.6]{Siburg-2021}}]\label{lem:delta_1}
If $\pi^n, \pi \in \mathcal{P}_c(\R^2)$ are SI, then 
\begin{align*}
    d_1(C_{\pi^n},C_{\pi})\to 0 \quad \Leftrightarrow \quad \|C_{\pi^n}- C_{\pi}\|_\infty\to 0.
\end{align*}
\end{lemma}

There are many ways to project a coupling $\pi\in\mathcal{P}(\R^2)$ onto $\cP^\uparrow(\R^2)$ without changing the marginals of $\pi.$ We highlight here one specific such projection, which goes back to \cite[Proposition 3.17]{Ansari-Rueschendorf-2021} (see also \cite{strothmann2022}), and will be particularly useful for our purposes.

\begin{definition}[Increasing rearrangement]\label{def:inc_rea}
\mbox{}
\begin{enumerate}[label = (\roman*)]
    \item For $\pi\in \mathcal{P}(\R^2)$, we denote by \(\pi^\uparrow \in \mathcal{P}(\R^2)\) the unique element of $\mathcal{P}^\uparrow(\R^2)$ satisfying $X\stackrel{d}{=} X^\uparrow$ and
\begin{align*}
    F_{Y|X}(y) \stackrel{d}{=} F_{Y^{\uparrow}|X^{\uparrow}}(y) \qquad \text{for all }y\in \R,\text{ where }(X,Y)\sim \pi, (X^\uparrow, Y^\uparrow)\sim \pi^\uparrow. 
\end{align*}
We call $\pi^\uparrow$ the \emph{increasing rearrangement} of $\pi.$
\item  We define the increasing rearrangement $C^\uparrow$ of a copula $C$ as the cdf of the increasing rearrangement of the distribution $\pi\in \Pi(\mathcal{U},\mathcal{U})$ induced by $C$.
\end{enumerate}
\end{definition}

It turns out that increasing rearrangements of couplings and copulas commute in the following sense.
\begin{lemma}\label{lem:commute}
 For \(\pi\in\cP_c(\R^2)\) we have
 \begin{align}
     C_{\pi^\uparrow} = C_\pi^\uparrow.
 \end{align}
\end{lemma}

\begin{proof}
Note that $\pi$ and $\pi^\uparrow$ have the same marginals. Then the claim follows from \eqref{eq:sklar}.
\end{proof}

We use the following properties of increasing rearranged copulas.

\begin{lemma}\label{lem_char_incRearr}
    For \(\pi\in \cP_c(\R^2)\), we have:
    \begin{enumerate}[label = (\roman*)]
        \item \label{lem_char_incRearr1} Characterization of independence: \(C_\pi^\uparrow = C^\perp\) if and only if \(\pi = \pi_1\otimes \pi_2\),
        \item \label{lem_char_incRearr2} Characterization of perfect dependence: \(C_\pi^\uparrow = C^+\) if and only if \(\pi = (x,f(x))_\#\pi_1\) for some measurable function \(f: \R\to \R\),
        \item \label{lem_char_incRearr3} Positive dependence: \(C_\pi^\uparrow\) is SI and thus \(C^\perp(u,v) \leq C_\pi^\uparrow(u,v) \leq C^+(u,v)\) for all \((u,v)\in [0,1]^2\).
    \end{enumerate}
\end{lemma}

We provide a proof of this essentially well-known result in the appendix.

\section{Comparisons of metrics and topologies}\label{sec:comparison}

In this section we compare the metrics and induced topologies introduced in Section \ref{sec:prelim}. We restrict our analysis to the sets  $\Pi(\mathcal{U}, \mathcal{U})\subset \mathcal{P}_c(\R^2)$ and $ \Pi^\uparrow(\mathcal{U}, \mathcal{U})=\Pi(\mathcal{U}, \mathcal{U})\cap \mathcal{P}^\uparrow(\R^2)$, on which we have a direct correspondence between copulas and couplings.

\subsection{Comparisons on $\Pi(\mathcal{U}, \mathcal{U})$}

Throughout this section, we use the identification of $\pi\in\Pi(\mathcal{U}, \mathcal{U})$ with its copula $C_\pi$, introduced in Section \ref{sec:prelim}. We first recall a basic fact: 

\begin{lemma}\label{lem:w1}
For any $\pi, \pi^n\in \Pi(\mathcal{U}, \mathcal{U})$ we have
\begin{align*}
    \mathcal{W}(\pi^n, \pi)\to 0 \quad \Leftrightarrow \quad C_{\pi_n} \to C_\pi \text{ pointwise } \quad \Leftrightarrow \quad  \|C_{\pi^n} - C_\pi\|_\infty \to 0.
\end{align*}
\end{lemma}
\begin{proof}
    First note that $\Pi(\mathcal{U}, \mathcal{U})\subset \mathcal{P}([0,1]^2)$, so Wasserstein convergence is the same as weak convergence; see \cite[Cor. 6.13]{Villani-2009}. As all couplings have uniform marginals, the claim follows from Lemma \ref{lem:polya}.
\end{proof}
In addition to \eqref{eq: general_inequalities}, we can then compare $\mathcal{KR}$ and $d_1$ on $\Pi(\mathcal{U}, \mathcal{U})$ as follows:
\begin{lemma}\label{lem:KR}
For any $\pi, \tilde \pi\in \Pi(\mathcal{U}, \mathcal{U})$ we have
\begin{align}\label{eq: W vs d_1}
    \mathcal{KR}(\pi, \tilde \pi)=  d_1(C_\pi, C_{\tilde\pi}).
\end{align}
\end{lemma}

\begin{proof}
As $\pi_1=\tilde \pi_1$ we have for $(X,Y)\sim \pi,$ $(\widetilde X, \widetilde Y)\sim \tilde \pi$ that
\begin{align*}
    \mathcal{KR}(\pi, \tilde\pi)&= \int_{[0,1]^2} |F_{Y|{X=u}}^{-1}(v) - F_{\widetilde Y|\widetilde X=u}^{-1}(v)|\,\de v \,\de u\\
    &=  \int_{[0,1]^2} |F_{Y|{X=u}}(v) - F_{\widetilde Y|\widetilde X=u}(v)|\,\de v \,\de u\\
    &\stackrel{\eqref{eq:d1}}{=}\int_{[0,1]^2} |\partial_1 C_{\pi}(u,v) - \partial_1 C_{\tilde{\pi}}(u,v)| \,\de v\, \de u = d_1(C_{\pi},C_{\tilde{\pi}}),
\end{align*}
where the second equality holds due to Lemma \ref{lem:L1_integral}.
\end{proof}
From this we directly conclude that the weak adapted, Knothe-Rosenblatt and $\partial_1$-topology are equal on $\Pi(\mathcal{U}, \mathcal{U}).$
\begin{corollary}\label{cor:d1}
For any $\pi^n,\pi\in \Pi(\mathcal{U}, \mathcal{U})$ we have 
\begin{align*}
    \mathcal{AW}(\pi^n, \pi)\to 0 \quad \Leftrightarrow \quad \mathcal{KR}(\pi^n, \pi)\to 0 \quad \Leftrightarrow \quad d_1(C_{\pi^n},C_\pi)\to 0. 
\end{align*}
\end{corollary}
\begin{proof}
    The first equivalence was stated in Lemma \ref{lem:beiglboeck-pammer}. The second equivalence follows from Lemma \ref{lem:KR}.
\end{proof}

While $ \mathcal{AW},  \mathcal{KR}$ and $d_1$ are thus compatible, the situation is different for $\mathcal{AW}, \mathcal{W}$, and $ \mathcal{CW}$. As we will see below, each one of these distances generates a different topology. In fact, for any $\pi^n,\pi \in \Pi(\mathcal{U}, \mathcal{U})$, we have
\begin{align*}
    \mathcal{AW}(\pi^n, \pi) \to 0 \quad \Rightarrow \quad \mathcal{W}(\pi^n, \pi)\to 0,\,\mathcal{CW}(\pi^n, \pi)\to 0
\end{align*}
by \eqref{eq: general_inequalities}. However, there exist $\pi^n,\pi \in \Pi(\mathcal{U}, \mathcal{U})$ such that
\begin{align*}
    \mathcal{CW}(\pi^n, \pi)\to 0 \quad \text{but} \quad \mathcal{AW}(\pi^n, \pi)\not\to 0, \,\mathcal{W}(\pi^n, \pi)\not\to 0,
\end{align*}
as the following example shows.

\begin{example} Recall the measures $\pi, \tilde \pi \in \Pi(\mathcal{U}, \mathcal{U})$ from Example \ref{ex:pseudo} and set $\pi^n=\tilde \pi$ for all $n\in \N.$ Following the same arguments as in Example \ref{ex:pseudo} we have $\mathcal{CW}(\pi^n, \pi) =0$ for all \(n\in \N\). But $\pi^n=\tilde \pi \neq \pi$, and as $\mathcal{W}$, $\mathcal{AW}$ are metrics, we conclude $\mathcal{AW}(\pi^n,\pi) \not\to 0, $ $\mathcal{W}(\pi^n, \pi)\not\to 0.$
\end{example}
Furthermore, there exist $\pi^n,\pi \in \Pi(\mathcal{U}, \mathcal{U})$ such that
\begin{align*}
     \mathcal{W}(\pi^n, \pi)\to 0 \quad \text{but}  \quad \mathcal{CW}(\pi^n, \pi)\not\to 0, \,\mathcal{AW}(\pi^n, \pi) \not\to 0.
\end{align*}
as the following classical example shows.

\begin{example}\label{ex:shuffle-of-min}
For $\pi^n= (u,nu-\lfloor nu \rfloor)_{\#}\mathcal{U}$ and $\pi=\mathcal{U}\otimes \mathcal{U}\in \Pi(\mathcal{U}, \mathcal{U})$, we compute 
\begin{align*}
\mathcal{CW}(\pi^n, \pi)&=\inf_{\gamma\in \Pi(\cU,\cU)} \int_{[0,1]^2} \cW(\pi^n_u,\pi_{\tilde{u}})\, \gamma(\de u,\de \tilde{u})  =  \int_{[0,1]} \cW(\pi^n_u,\mathcal{U}) \, \de u =\int_{[0,1]^2} |u-\tilde u| \, \de u\, \de \tilde u =1/3\neq 0
\end{align*}
for all $n\in \N.$
On the other hand, $\mathcal{W}(\pi^n, \pi) \to 0,$
as for any bounded, continuous function $f\colon \R^2\to \R$, 
\begin{align*}
\int_{[0,1]^2} f(u, \tilde u)\,\pi^n(\de u,\de \tilde u) = \int_{[0,1]} f(u,nu-\lfloor nu \rfloor)\, \de u \to \int_{[0,1]^2} f(u,\tilde u)\,\de u \,\de \tilde u.
\end{align*}
On the level of copulas, the above construction corresponds to the concept of gluing copulas \cite{Siburg-2008}. Related constructions are based on shuffle-of-min copulas which are dense in the set of copulas with respect to uniform convergence \cite{Kimeldorf-1978,Mikusinski-Sherwood-Taylor-1992}.
\end{example}

We summarize the situation in Figure \ref{fig:uu} below. 

\begin{figure}[h!] 
\begin{tikzpicture}[
    font=\large,
    line cap=round,
    line join=round,
    >=Latex,
    node distance=1cm,
    dbl/.style={
        draw=black,
        double=black,
        double distance=0.9pt,
        line width=0.45pt
    },
    niceleft/.style={
        double,
        line width=0.4pt,
        line cap=round,
  line join=round,
         {Latex[open,length=1.5mm,width=1.5mm]}-
    },
    niceright/.style={
        double,
        line width=0.4pt,
        line cap=round,
  line join=round,
         -{Latex[open,length=1.5mm,width=1.5mm]}
    },
    niceboth/.style={
        double,
        line width=0.4pt,
  line cap=round,
  line join=round,
  {Latex[open,length=1.5mm,width=1.5mm]}-{Latex[open,length=1.5mm,width=1.5mm]}
    }]

% Main nodes
\node (AW) {$\mathcal{AW}(\pi^n, \pi)\to 0$};
\node (KR) [right=0.75cm of AW] {$\mathcal{KR}(\pi^n, \pi)\to 0$};
\node (d1) [right=0.75cm of KR] {$d_1(C_{\pi^n}, C_\pi)\to 0$};

% W and bold W positioned slightly above/below and to the right of AW
\node (W)  [above=0.5cm of AW] {$\mathcal{W}(\pi^n, \pi)\to 0$};
\node (bW) [below=0.5cm of AW] {$\mathcal{CW}(\pi^n, \pi)\to 0$};
\node (U) [left=0.5cm of W] {$\|C_{\pi^n}-C_\pi\|_\infty \to 0$};

% Relations
\draw[niceleft] (W) -- (AW);   % W <- AW
\draw[niceleft] (bW) -- (AW);   % W <- AW
\draw[niceboth]                  (AW) -- (KR);  % AW <-> KR
\draw[niceboth]                  (KR) -- (d1);  % KR <-> d_1
\draw[niceboth]                  (W) -- (U);  % KR <-> d_1

% Optional vertical relation between W and bold W

\end{tikzpicture}
\caption{Topologies on $\Pi(\mathcal{U}, \mathcal{U})$.}
\label{fig:uu}
\end{figure}

\subsection{Comparisons on {$\Pi^{\uparrow}(\mathcal{U}, \mathcal{U})$}}

The situation changes when restricting the setting to $\Pi^{\uparrow}(\mathcal{U}, \mathcal{U}).$ Indeed, the Knothe-Rosenblatt coupling is the optimal coupling of any \(\pi, \tilde \pi \in \Pi^\uparrow(\cU,\cU)\) according to Lemma \ref{lem_rep_IW_AW}. This leads to the following lemma:

\begin{lemma} \label{lem:si}
For $\pi, \tilde \pi \in \Pi(\mathcal{U}, \mathcal{U})$ we have
\begin{align*}
    \mathcal{CW}(\pi, \tilde \pi)\le \mathcal{AW}(\pi, \tilde \pi) =\mathcal{KR}(\pi, \tilde \pi)= d_1(C_\pi, C_{\tilde\pi}),
\end{align*}
if the Knothe-Rosenblatt coupling is optimal for $\mathcal{AW}$. If $\pi$ and $\tilde \pi$ are SI, then 
\begin{align*}
    \mathcal{CW}(\pi, \tilde \pi)= \mathcal{AW}(\pi, \tilde \pi) =\mathcal{KR}(\pi, \tilde \pi)=  d_1(C_\pi, C_{\tilde\pi}).
\end{align*}
\end{lemma}
\begin{proof}
    If the Knothe-Rosenblatt coupling is optimal for $\mathcal{AW}$, then we have 
    \begin{align}\label{eq:kr_proof}
        \mathcal{AW}(\pi, \tilde\pi)=\mathcal{KR}(\pi, \tilde \pi) = \int_{[0,1]^2} |F^{-1}_{Y|X=u}(v)-F^{-1}_{\widetilde Y|\widetilde X=u}(v)|  \, \de v\, \de u =d_1(C_\pi, C_{\tilde \pi}),
    \end{align}
    recalling the proof of Lemma \ref{lem:KR}. If $\pi$ and $\tilde \pi$ are SI, then the Knothe-Rosenblatt coupling is optimal for $\mathcal{CW}$ and $\mathcal{AW}$ by Lemma \ref{lem_rep_IW_AW}, so the claim follows from \eqref{eq:kr_proof}, \eqref{eq:bw_si} and \eqref{eq:aw_si}.
\end{proof}

In particular all topologies are equal on $\Pi^{\uparrow}(\mathcal{U}, \mathcal{U})$.

\begin{corollary}\label{cor:si}
If $\pi^n, \pi\in \Pi^{\uparrow}(\mathcal{U}, \mathcal{U})$, then the following are equivalent:
\begin{enumerate}[label = (\roman*)]
\item $\mathcal{W}(\pi^n,\pi)\to 0$,
\item $\mathcal{CW}(\pi^n,\pi)\to 0$,
\item $\mathcal{AW}(\pi^n,\pi)\to 0$,
\item $\mathcal{KR}(\pi^n,\pi)\to 0$,
\item $d_1(C_{\pi^n},C_\pi)\to 0$,
\item $\|C_{\pi^n} - C_\pi\|_\infty\to 0$.
\end{enumerate}
\end{corollary}

\begin{proof}
Recall that $(v)\Leftrightarrow (vi)$ holds by Lemma \ref{lem:delta_1}. The equivalences $(v) \Leftrightarrow (iv) \Leftrightarrow (iii) \Leftrightarrow (ii)$ follow from Lemma \ref{lem:si}, $(iii) \Rightarrow (i)$ by \eqref{eq: general_inequalities}, and $(i) \Rightarrow (vi)$ by Lemma \ref{lem:w1}.
\end{proof}

Figure \ref{fig:uparrow} summarizes these equivalences.

\begin{figure}[h!] 
\begin{tikzpicture}[
    font=\large,
    line cap=round,
    line join=round,
    >=Latex,
    node distance=1.5cm,
    dbl/.style={
        draw=black,
        double=black,
        double distance=0.9pt,
        line width=0.45pt
    },
    niceleft/.style={
        double,
        line width=0.4pt,
        line cap=round,
  line join=round,
         {Latex[open,length=1.5mm,width=1.5mm]}-
    },
    niceright/.style={
        double,
        line width=0.4pt,
        line cap=round,
  line join=round,
         -{Latex[open,length=1.5mm,width=1.5mm]}
    },
    niceboth/.style={
        double,
        line width=0.4pt,
  line cap=round,
  line join=round,
  {Latex[open,length=1.5mm,width=1.5mm]}-{Latex[open,length=1.5mm,width=1.5mm]}
    }]

% Main nodes
\node (W)  {$\mathcal{W}(\pi^n, \pi)\to 0$};
\node (AW) [below=0.5cm of W] {$\mathcal{AW}(\pi^n, \pi)\to 0$};
\node (bW) [below=0.5cm of AW] {$\mathcal{CW}(\pi^n, \pi)\to 0$};
\node (KR) [right=0.75cm of AW] {$\mathcal{KR}(\pi^n, \pi)\to 0$};
\node (d1) [right=0.75cm of KR] {$d_1(C_{\pi^n}, C_{\pi})\to 0$};
\node (U) [above=0.5cm of d1] {$\|C_{\pi^n}- C_{\pi}\|_\infty\to 0$};

% W and bold W positioned slightly above/below and to the right of AW

% Relations
\draw[niceboth] (W) -- (AW);   % W <- AW
\draw[niceboth] (AW) -- (bW);   % W <- AW
\draw[niceboth]                  (AW) -- (KR);  % AW <-> KR
\draw[niceboth]                  (KR) -- (d1);  % KR <-> d_1
\draw[niceboth]                  (d1) -- (U);  % KR <-> d_1

% Optional vertical relation between W and bold W

\end{tikzpicture}
\caption{Topologies on $\Pi^{\uparrow}(\mathcal{U}, \mathcal{U})$}
\label{fig:uparrow}
\end{figure}

\subsection{Metric comparisons of couplings and their SI rearrangements}

Now we want to connect couplings in $\cP(\R^2)$ to couplings in $\cP^\uparrow(\R^2)$. We achieve this through the increasing rearrangements from Definition \ref{def:inc_rea}. In fact, it turns out that the adapted distance between increasing rearrangements of $\pi$ and $\tilde \pi$ is always dominated by the adapted distance between $\pi$ and $\tilde\pi$, as shown below.

\begin{proposition}\label{prop:comparison}
For any couplings $\pi, \tilde\pi \in \mathcal{P}(\R^{2})$ we have
\begin{align*}
\mathcal{AW}(\pi^{\uparrow}, \tilde\pi^{\uparrow}) \le \mathcal{AW}(\pi, \tilde \pi) \quad \text{ and }\quad 
\mathcal{CW}(\pi^{\uparrow}, \tilde\pi^{\uparrow}) \le \mathcal{CW}(\pi, \tilde \pi).
\end{align*}
\end{proposition}

\begin{proof}
Let $(X,Y) \sim \pi$ and $(\widetilde X, \widetilde Y) \sim \tilde \pi$.
For \(y\in \R\), consider the random variables 
\begin{align*}
    F_{Y|X}(y)  \text{ and }  F_{\widetilde{Y}|\widetilde X}(y).
\end{align*}
Then, for the increasing rearrangements $\pi^\uparrow, \tilde\pi^\uparrow$ of $\pi, \tilde \pi,$ let $(X^\uparrow, Y^\uparrow)\sim \pi^\uparrow,(\tilde{X}^\uparrow,\widetilde{Y}^\uparrow)\sim \tilde{\pi}^\uparrow$ and consider the random variables 
\begin{align*}
    F_{Y^\uparrow|X^\uparrow}(y)  \text{ and } F_{\widetilde{Y}^\uparrow|\widetilde{X}^\uparrow}(y). 
\end{align*}
Recalling Definition \ref{def:inc_rea} we have \(F_{Y|X}(y)\eqd F_{Y^\uparrow|X^\uparrow}(y)\), \(F_{\widetilde{Y}|\widetilde X}(y) \eqd F_{\widetilde{Y}^\uparrow|\widetilde{X}^\uparrow}(y)\), and \(F_{Y^\uparrow|X^\uparrow}(y),F_{\widetilde{Y}^\uparrow|\widetilde{X}^\uparrow}(y)\) are comonotone for any $y\in \R$. %As in the proof of \cite[Theorem 3.6\,(i)]{Ansari-2026} we now conclude that
This yields
    \begin{align}\label{ineq_cx}
        F_{Y|X}(y) - F_{\widetilde{Y}|\widetilde X}(y) \geq_{cx} F_{Y^\uparrow|X^\uparrow}(y) - F_{\widetilde{Y}^\uparrow|\widetilde{X}^\uparrow}(y);
    \end{align}
    see e.g.~\cite[Corollary 3.28\,(b)]{Ru-2013}. %\textcolor{red}{Why are there 2 references here, to Ru and Ans?}. 
    Thus we obtain
    \begin{align*}
        \mathcal{AW}(\pi,\tilde{\pi}) &= \inf_{\gamma \in \Pi(\pi_1,\tilde{\pi}_1)} \int |x-\tilde{x}| + \mathcal{W}(\pi_x,\tilde{\pi}_{\tilde{x}}) \,\gamma(\de x,\de \tilde{x}) \\
        &\geq  \inf_{\gamma \in \Pi(\pi_1,\tilde{\pi}_1)} \int |x-\tilde{x}| \,\gamma(\de x, \de \tilde{x}) \\
        &\quad + \inf_{\gamma \in \Pi(\pi_1,\tilde{\pi}_1)} \int \int_0^1 \big|F_{Y|X=x}^{-1}(v) - F_{\tilde{Y}|\tilde{X}=\tilde{x}}^{-1}(v) \big|\, \de v \,\gamma(\de x, \de \tilde{x}) \\
        &= \cW(\pi_1,\tilde{\pi}_1) + \inf_{\gamma \in \Pi(\pi_1,\tilde{\pi}_1)}  \int  \big|F_{Y|X=x}(y) - F_{\tilde{Y}|\tilde{X}=\tilde{x}}(y) \big|\, \de y \,\gamma(\de x, \de \tilde{x}) \\
        & \geq \cW(\pi_1,\tilde{\pi}_1) + \E \Big[\int  \big|F_{Y^\uparrow|X^\uparrow }(y) - F_{\widetilde{Y}^\uparrow|\widetilde{X}^\uparrow}(y) \big| \,\de y\Big] \\
        &= \int_0^1 \big|F_{X}^{-1}(u) - F_{\tilde{X}}^{-1}(u)\big| +\int_0^1 \big|F_{Y^\uparrow|X^\uparrow=F_X^{-1}(u)}^{-1}(v) - F_{\widetilde{Y}^\uparrow|\tilde{X}^\uparrow=F_{\widetilde{X}}^{-1}(u)}^{-1}(v) \big| \,\de v \,\de u  \\
        &= \mathcal{AW}(\pi^\uparrow,\tilde{\pi}^\uparrow),
    \end{align*}
    where the first inequality is trivial, the second equality uses Lemma \ref{lem:L1_integral}, the second inequality is due to \eqref{ineq_cx} and the third equality uses again Lemma \ref{lem:L1_integral} and the fact that \(F_{Y^\uparrow|X^\uparrow}(y),F_{\widetilde{Y}^\uparrow|\widetilde{X}^\uparrow}(y)\) are comonotone for any $y\in \R$. The fourth equality follows from Lemma \ref{lem_rep_IW_AW} noting that \(X\stackrel{d}{=} X^\uparrow \sim \pi_1^\uparrow\) and \(\widetilde{X} \stackrel{d}{=} \widetilde{X}^\uparrow \sim \tilde{\pi}_1^\uparrow\). 
     Lastly, the inequality for $\mathcal{CW}$ follows by omitting the first term in the integrands.
\end{proof}

An immediate corollary of the above is the following result; see also \cite[Theorems 3.4\,(i) and 3.6\,(i)]{Ansari-2026} and \cite[Proof of Theorem 3.2]{strothmann2022}.

\begin{corollary}\label{cor_incrRear}
For any $\pi, \tilde \pi\in \Pi(\mathcal{U}, \mathcal{U})$ we have
\begin{align*}
    \mathcal{KR}(\pi^\uparrow, \tilde {\pi}^\uparrow)=d_1( C_{\pi}^{\uparrow}, C_{\tilde \pi}^{\uparrow}) = \mathcal{AW}(\pi^\uparrow, \tilde \pi^\uparrow) \le \mathcal{AW}(\pi, \tilde \pi)\le \mathcal{KR}(\pi, \tilde \pi)=d_1(C_\pi, C_{\tilde \pi}).
\end{align*}
\end{corollary}

\begin{proof}
From Lemmas \ref{lem:commute} and \ref{lem:si} we have $ d_1( C_{\pi}^{\uparrow}, C_{\tilde \pi}^{\uparrow})=\mathcal{KR}(\pi^\uparrow, \tilde \pi^\uparrow)=\mathcal{AW}(\pi^\uparrow, \tilde \pi^\uparrow)$ as $\pi^\uparrow, \tilde \pi^\uparrow$ are SI. From Proposition \ref{prop:comparison} it follows that $\mathcal{AW}(\pi^\uparrow, \tilde \pi^\uparrow) \le \mathcal{AW}(\pi, \tilde \pi)$. Next, by \eqref{eq: general_inequalities} $\mathcal{AW}(\pi, \tilde \pi) \le \mathcal{KR}(\pi, \tilde \pi)$. Finally $\mathcal{KR}(\pi, \tilde \pi)= d_1(C_\pi, C_{\tilde \pi})$ by Lemma \ref{lem:KR}. This concludes the proof.
\end{proof}

\begin{example}
    Proposition \ref{prop:comparison} fails to hold for the classical Wasserstein distance. To see this, recall from Example \ref{ex:shuffle-of-min} that there exists a sequence \(\pi^n = (x,f_n(x))_\# \pi^n_1\) in $\Pi(\mathcal{U}, \mathcal{U})$  such that  \(\cW(\pi^{n},\pi) \to 0\), where $\pi = \mathcal{U}\otimes \mathcal{U}$. However, by construction, for \((X_n,Y_n)\sim \pi_n\), we see that \(Y_n\) perfectly depends on \(X_n\). By Lemma \ref{lem_char_incRearr}, this yields \(C_{\pi_n}^\uparrow = C^+\) for all \(n\) and \(C_\pi^\uparrow = C^{\perp}\). It follows that \(\cW(\pi_{n}^\uparrow,\pi^\uparrow) \equiv \cW(\pi_{C^+},\pi_{C^\perp}) > 0\) for all \(n\).
\end{example}

\subsection{Continuity in bivariate models}

We provide some general conditions on bivariate, absolutely continuous distributions to ensure continuity under the adapted Wasserstein distance and, in particular, under the conditional Wasserstein distance.
First, we recall the following lemmas:
\begin{lemma}[{\cite[Lemma 37]{blanchet2024bounding}}]
    For all $R>0$ define the map $\Phi^R:\R^2 \to \R^2$ via $\Phi^R(x,y):=(\varphi^R(x), \varphi^R(y))$, where 
    \begin{align*}
        \varphi^R(x):=\begin{cases}
            x &\text{if }|x|< R,\\
            \frac{Rx}{|x|} &\text{ otherwise.}
        \end{cases}
    \end{align*}
    Then 
    \begin{align*}
        \mathcal{AW}(\pi,(\Phi^R)_\#\pi) \le 8 \int_{\{|x|+|y|\ge R\}} (|x|+|y|)\,\pi(\de x, \de y)
    \end{align*}
\end{lemma}

For the next lemma, recall that the total variation distance between two measures $\pi, \tilde\pi \in \mathcal{P}(\R^2)$ is given by
\begin{align*}
    \mathrm{TV}(\pi, \tilde\pi) =\sup_{A\subseteq \R^2} |\pi(A)-\tilde\pi(A)|.
\end{align*}

\begin{lemma}[{\cite[Cor. 2.7]{acciaio2025estimating}}]
    For any $\pi, \tilde \pi \in \mathcal{P}(\R^2)$ we have
    \begin{align*}
        \mathcal{AW}((\Phi^R)_\#\pi, (\Phi^R)_\#\tilde \pi)\le 3R \cdot \operatorname{TV}((\Phi^R)_\#\pi, (\Phi^R)_\#\tilde \pi).
    \end{align*}
\end{lemma}
As TV is contractive under pushforward maps we conclude $$\operatorname{TV}((\Phi^R)_\#\pi, (\Phi^R)_\#\tilde \pi) \le \operatorname{TV}(\pi, \tilde\pi).$$
Combining these facts leads to the following lemma:
\begin{lemma}
    For any $\pi, \tilde \pi \in \mathcal{P}(\R^2)$ and any $R>0$ we have
    \begin{align*}
        \mathcal{AW}(\pi, \tilde \pi) \le 8 \int_{\{|x|+|y|\ge R\}} (|x|+|y|)\,(\pi+\tilde \pi)(\de x, \de y) + 3R \cdot\operatorname{TV}(\pi, \tilde\pi).
    \end{align*}
\end{lemma}
This gives the following corollary:
\begin{corollary}
    Let $\pi \in \mathcal{P}(\R^2)$ and let $(\pi^n)$ be a sequence in $\mathcal{P}(\R^2)$ such that $\operatorname{TV}(\pi^n, \pi)\to 0$ and $(\pi^n)$ is $\mathcal{W}$-precompact, i.e.
    \begin{align}\label{eq:precompact}
        \lim_{R\to \infty} \sup_{n\in \N} \int_{\{|x|+|y|\ge R\}} (|x|+|y|)\,\pi^n (\de x, \de y)=0.
    \end{align}
    Then $\mathcal{AW}(\pi^n, \pi)\to 0.$
\end{corollary}
Note that a sufficient condition for \eqref{eq:precompact} is $\sup_{n\in \N} \int (|x|+|y|)^{1+s}\,\pi^n (\de x, \de y)<\infty$ for some $s>0$. This yields the following corollary:
\begin{corollary}\label{cor:continuity}
    Assume that $\pi^n,\pi$ are absolutely continuous w.r.t. the Lebesgue measure with densities $f_n, f$, and $f_n \to f$ a.e. as well as $\sup_{n\in \N} \int (|x|+|y|)^{1+s}\,\pi^n (\de x, \de y)<\infty$ for some $s>0$. Then $\mathcal{AW}(\pi^n,\pi)\to 0.$
\end{corollary}
Note that Corollary \ref{cor:continuity} applies in particular to absolutely continuous
\begin{itemize}
    \item Elliptical distributions, 
    \item Exponential distributions,
    \item Location-scale families of distributions,
    \item Copula-based parametric families of distributions.
\end{itemize}

\section{Estimation}\label{sec:estimation}

As discussed in the Introduction, many dependence measures satisfying Axioms \ref{axiom1}--\ref{axiom3} are defined via conditional distributions. Since the classical empirical copula process fails to estimate these distributions adequately, modified estimators based e.g. on binning methods are necessary. In this section, we first recall such a binning method as introduced in the adapted optimal transport literature. We then apply it to derive a new copula estimator. Subsequently we compare it to the rank-based checkerboard estimator. It will turn out that both estimators have the same rates of convergence under regularity assumptions.

\subsection{Setting and preliminary estimates} \label{sec:1}

In Sections \ref{sec:1} and \ref{sec:2} we consider $\pi\in \mathcal{P}_c([0,1]^2)$. For an i.i.d.~sample \((X_l,Y_l)_{l=1}^N\) from \(\pi\), we denote the \emph{adapted empirical measure} by
\begin{align}\label{eq:adapted_empirical}
    \hat{\pi}^N := \frac 1 N \sum_{l=1}^N \delta_{(\varphi^N(X_l),\varphi^N(Y_l))}
\end{align}
for each \(N\geq 1\); see \cite{Backhoff-Bartl-Beiglboeck-Wiesel-2022}. In \eqref{eq:adapted_empirical}, the binning function \(\varphi^N\colon [0,1] \to [0,1]\) is defined as follows: recalling\footnote{Assuming for notational simplicity that $N^{1/3}$ is an integer; otherwise we replace it by \(\lfloor N^{1/3}\rfloor\).} \eqref{eq:square}, for $n:=N^{1/3}$, take a partition of \([0,1]^2\) into the disjoint union of the cubes
\begin{align*}
S_{ij}^n=\Big(\frac{i-1}{n}, \frac{i}{n} \Big] \times \Big(\frac{j-1}{n}, \frac{j}{n} \Big],\qquad i,j=1, \dots, n,
\end{align*}
with edge length \(N^{-1/3}\). For the intervals $S_i^n:= ((i-1)/n, i/n]$, $i=1, \dots,n$, the function $\varphi^N$ has finite range denoted by $(\varphi^N(S_1^n), \dots, \varphi^N(S_{n}^n))$ and maps each cube $S_{ij}^n$ to its center $(\varphi^N(S_i^n), \varphi^N(S_j^n))$.\footnote{We make the convention  that $\varphi^N(0)=1/(2n).$} By definition we have
\begin{align}\label{subcube_bound}
    \sup_{u\in [0,1]} |u - \varphi^N(u)| \leq c N^{-1/3}.
\end{align}
The adapted empirical measure is a strongly consistent estimator in adapted Wasserstein distance, i.e.~
\begin{align}\label{asconv_AEM}
    \cAW(\hat{\pi}^N, \pi) \to 0 \qquad \P\text{-almost surely};
\end{align}
see \cite[Theorem 1.3]{Backhoff-Bartl-Beiglboeck-Wiesel-2022}. If $\pi$ has regular kernels, we can also obtain a convergence speed.

\begin{definition}\label{def:lipschitz_kernel}
We say that a probability measure $\pi\in \mathcal{P}([0,1]^2)$ has \emph{Lipschitz kernels}, if there exists a version of the kernel \((\pi_x)\) such that the mapping
\begin{align}\label{ass_Lip_kernel}
   [0,1]\ni  x\mapsto \pi_x \in \mathcal{P}([0,1]) \quad \text{ is Lipschitz-continuous}
\end{align}
with respect to the Wasserstein distance $\mathcal{W}$ on \(\mathcal{P}([0,1])\). 
\end{definition}
In fact, if $\pi$ has Lipschitz kernels \eqref{ass_Lip_kernel}, then 
\begin{align}\label{eq:lipschitz_kernel_est}
    \E\left[\cAW(\hat{\pi}^N,\pi)\right] \leq c N^{-1/3};
\end{align}
see \cite[Theorem 1.5]{Backhoff-Bartl-Beiglboeck-Wiesel-2022}.

In order to apply the results from Section \ref{sec:comparison}, we need to modify $\hat\pi^N$ in such a way that it has continuous marginal distributions. 
We thus denote by $\hat{\pi}_c^N$ a continuous modification of $\hat{\pi}^N$, obtained by uniformly distributing the mass concentrated at each center over its corresponding subcube, i.e.
\begin{align}\label{eq:def_cont}
    \hat{\pi}_c^N := \sum_{i,j=1}^{n}  \hat{\pi}^N(S_{ij}^n) \cdot \mathcal{U}^{S_{ij}^n}.
\end{align}
We call $\hat{\pi}_c^N$ the \emph{continuous adapted empirical measure.}
Note that $\hat{\pi}_c^N$ has continuous marginals (that are not necessarily equal to $\mathcal{U}$) and, recalling \eqref{eq:sklar}, thus possesses a unique induced copula that we denote by
\begin{align}\label{def_cop_est}
    \hat{C}_N := C_{\hat{\pi}^N_c}.
\end{align}
We call $\hat{C}_N$ the \emph{adapted copula estimator}; see also Algorithm \ref{alg:ACE}.
\begin{algorithm}[t]
\caption{Construction of the adapted copula estimator}
\label{alg:ACE}
\begin{flushleft}
\textbf{Data:} An i.i.d.\ sample $(X_l, Y_l)_{l=1}^N$ from $\pi \in \mathcal{P}_c([0,1]^2)$\\
\textbf{Result:} The adapted copula estimator $\widehat{C}_N$
\end{flushleft}

\begin{enumerate}[label=(\arabic*), leftmargin=*]
    \item Set $n := \max(1, \lfloor N^{1/3} \rfloor)$ to define the grid resolution.
    
    \item For every $k = 1, \dots, N$, determine the corresponding subcube indices:\\[1mm]
    $\displaystyle i_k := \max(1,\lceil n X_k \rceil) , \quad j_k := \max(1,\lceil n Y_k \rceil).$
    
    \item Construct the empirical mass matrix $P \in \mathbb{R}^{n \times n}$ with entries\\[1mm]
    $\displaystyle P_{i,j} := \frac{1}{N} \sum_{k=1}^N \mathbf{1}_{\{i_k = i, j_k = j\}} \quad \text{for } i,j = 1, \dots, n.$
    
    \item Define the continuous adapted empirical measure $\hat{\pi}_c^N$ by uniformly distributing the mass $P_{i,j} = \hat{\pi}_N(S_{ij}^n)$ over the disjoint cubes $S^n_{ij} := \left(\frac{i-1}{n}, \frac{i}{n}\right] \times \left(\frac{j-1}{n}, \frac{j}{n}\right]$:\\[1mm]
    $\displaystyle \hat{\pi}_N^c := \sum_{i,j=1}^n P_{i,j} \cdot \mathcal{U}^{S_{ij}},$\\[1mm]
    where $\mathcal{U}^{S_{ij}}$ denotes the uniform distribution on $S_{ij}$.
    
    \item Define the adapted copula estimator $\widehat{C}_N$ as the unique copula induced via Sklar's theorem by $\hat{\pi}_c^N$:\\[1mm]
    $\displaystyle \widehat{C}_N := C_{\widehat{\pi}_c^N} = (F_{(\hat{\pi}_c^N)_1},F_{(\hat{\pi}_c^N)_2})_\#\hat{\pi}_c^N.$
\end{enumerate}
\end{algorithm}
Let us also recall that  \(C_\pi\) is the copula associated with $\pi$. We are now in position to investigate convergence of $\hat{\pi}^N_c$ to $\pi$ and $\hat{C}_N$ to $C_\pi.$ We start with the following strengthening of \eqref{asconv_AEM} and \eqref{eq:lipschitz_kernel_est}.
\begin{theorem} \label{lem:triangle_est}
We have
\begin{align}\label{asconv_AEM_c}
\mathcal{KR}(\hat{\pi}^N_c,\pi) \leq \mathcal{KR}(\hat{\pi}^N,\pi) + c N^{-1/3} \to 0 \qquad \P\text{-almost surely}.
\end{align}
Furthermore, if $\pi$ has Lipschitz kernels as defined in \eqref{ass_Lip_kernel}, we have
\begin{align}\label{eq_rate_conv_c}
    \E\left[\mathcal{KR}(\hat{\pi}^N_c,\pi)\right] \leq c N^{-1/3}.
\end{align}
\end{theorem}

\begin{proof}
Using the triangle inequality for $\mathcal{KR}$ we obtain
\begin{align*}
    \mathcal{KR}(\hat{\pi}^N_c,\pi) \leq \mathcal{KR}(\hat{\pi}^N_c,\hat{\pi}^N) +\mathcal{KR}(\hat{\pi}^N,\pi).
\end{align*}
In order to estimate $\mathcal{KR}(\hat{\pi}^N_c,\hat{\pi}^N)$, we note that 
\begin{align*}
    \hat{\pi}^N := \sum_{i,j=1}^{N^{1/3}} \hat{\pi}^N(S_{ij}^n) \cdot \delta_{(\varphi^N(S_i^n),\varphi^N(S_j^n))} .
\end{align*}
Let now $(X^N,Y^N)\sim \hat{\pi}^N_c, (\widetilde X^N , \widetilde Y^N)\sim \hat{\pi}^N.$
Recalling $S_i^n= ((i-1)/n, i/n]$, $i=1, \dots,n$, $\varphi^N(S_i^n)\in S_i^n$ and \eqref{eq:def_cont}, note that 
\begin{align}\label{eq:exact}
F_{X^N}^{-1}(u) \in S_i \quad \Leftrightarrow \quad F_{\widetilde X^N}^{-1}(u) \in S_i
\end{align}
for all $u\in [0,1]$ and $i=1, \dots, n.$ Thus
%the coupling $$\gamma(\de x, \de \tilde x)= \sum_{i=1}^{N^{1/3}} \hat{\pi}^N( S_i^n\times [0,1]) \cdot (\delta_{\varphi^N(S_i^n)}\otimes \mathcal{U}(S_{i}))(\de x, \de \tilde x)$$
%and note that $\gamma \in \Pi(\hat{\pi}^N_1, (\hat{\pi}^N_c)_1)$. We can thus use $\gamma$ to estimate the adapted Wasserstein distance via 
\begin{align*}
    \mathcal{KR}(\hat{\pi}^N_c,\hat{\pi}^N)&= \int \big|F_{X^N}^{-1}(u)-F_{\widetilde X^N}^{-1}(u)\big| +\mathcal{W}\Big((\hat{\pi}^N_c)_{F_{X^N}^{-1}(u)}, (\hat{\pi}^N)_{F_{\widetilde X^N}^{-1}(u)}\Big)\,\de u\\
    &\stackrel{\eqref{eq:exact},\eqref{subcube_bound}}{\le} c N^{-1/3}  +\int \mathcal{W}((\hat{\pi}^N_c)_x, (\hat{\pi}^N)_{x})\, \hat{\pi}^N(\de x),
\end{align*}
where the second inequality follows from the fact that both $x\mapsto(\hat{\pi}^N_c)_x$ and $x\mapsto (\hat{\pi}^N)_{x}$ are constant on each $S_i^n,$ $i=1, \dots, n$ by \eqref{eq:adapted_empirical} and \eqref{eq:def_cont}. Next we note that $\mathcal{W}((\hat{\pi}^N_c)_x, (\hat{\pi}^N)_{x})\le cN^{-1/3}$ again by \eqref{subcube_bound}.
It thus remains to estimate $\mathcal{KR}(\hat{\pi}^N, \pi)$. This can be done by the  methods of \cite{Backhoff-Bartl-Beiglboeck-Wiesel-2022}, noting that \cite[Lemma 3.1, Lemma 5.1]{Backhoff-Bartl-Beiglboeck-Wiesel-2022} still hold when $\mathcal{AW}$ is replaced by $\mathcal{KR}.$ The remainder of the proof in \cite{Backhoff-Bartl-Beiglboeck-Wiesel-2022} remains unchanged. In particular this shows $\E[\mathcal{KR}(\hat{\pi}^N, \pi)]\le cN^{-1/3}$ holds if $\pi$ has Lipschitz kernels. The claim follows. 
\end{proof}

\subsection{Empirical convergence and rates}\label{sec:2}

As stated in \eqref{asconv_AEM_c} we have  $\mathcal{KR}(\hat\pi_c^N, \pi)\to 0$. Then by Proposition \ref{prop:comparison}
\begin{align}\label{eq:adapted_rearranged}
    \mathcal{KR}((\hat\pi^N_c)^\uparrow, \pi^\uparrow)\le \mathcal{KR}(\hat\pi^N_c, \pi) \to 0
\end{align}
where $(\hat\pi^N_c)^\uparrow$, $\pi^\uparrow$ are the increasing rearrangements of $\hat\pi^N_c$ and $\pi$ respectively. Similarly, 
\begin{align}\label{eq:adapted_rearranged2}
     \E[\mathcal{KR}((\hat\pi^N_c)^\uparrow, \pi^\uparrow)]\le \E[\mathcal{KR}(\hat\pi^N_c, \pi)] \le cN^{-1/3}
\end{align}
follows from \eqref{eq_rate_conv_c}, if $\pi$ has Lipschitz kernels \eqref{ass_Lip_kernel}. Next we define 
\begin{align}\label{def_Incr_Rearr_Adap_Emp_Cop}
    \hat{C}_N^\uparrow:= C_{(\hat{\pi}_c^N)^\uparrow} =C_{\hat{\pi}_c^N}^\uparrow,
\end{align}
recalling Lemma \ref{lem:commute}. Then we obtain the following result. 
%Equation \eqref{eq:adapted_rearranged} then implies the following.
\begin{theorem}\label{thm_conv_C_N}
The adapted copula estimator \(\hat{C}_N\) satisfies $\P$-almost surely:
    \begin{enumerate}[label = (\roman*)]
        \item \label{thm_conv_C_N1} \(\|\hat{C}_N -C_\pi\|_\infty\to 0\),
        \item \label{thm_conv_C_N3} \(\|\hat{C}_N^\uparrow- C^\uparrow_\pi\|_\infty\to 0\),
        \item \label{thm_conv_C_N2} \(d_1(\hat{C}_N^\uparrow, C_\pi^\uparrow)\to 0\),
    \end{enumerate}
If $y \mapsto F_Y(y)$ is Lipschitz continuous for $Y\sim \pi_2$, then we also have 
\begin{enumerate}[label = (iv)]
    \item \(d_1(\hat{C}_N ,C_\pi)\to 0\).
\end{enumerate}
\end{theorem}

%\begin{remark}
In Section \ref{sec:comp_checkerboard} we discuss in more detail, how $\hat C_N$ compares to classical estimation techniques in dependence modeling and in particular to the checkerboard estimator.
%\end{remark}

\begin{proof}[Proof of Theorem \ref{thm_conv_C_N}.]
    Let us first point out that results from Section \ref{sec:comparison} are not directly applicable as $\hat{\pi}^N_c, (\hat{\pi}^N_c)^\uparrow\notin \Pi(\mathcal{U}, \mathcal{U}).$ 
    We first show $(i)$: equation \eqref{asconv_AEM_c} implies in particular $\mathcal{W}(\hat{\pi}_c^N,\pi)\to 0$ $\P$-almost surely. Since \(\hat{\pi}_c^N\) and \(\pi\) have continuous distribution functions, Lemma \ref{lem:polya} yields $(i)$. For $(ii)$ we argue as in $(i)$ noting that \eqref{eq:adapted_rearranged} implies \(\mathcal{W}( \hat{\pi}_c^{N\uparrow},\pi^\uparrow)\to 0.\) Statement $(iii)$ follows from Lemma \ref{lem:delta_1}, as $(\hat{\pi}_c^N)^\uparrow, \pi^\uparrow$ are SI. For $(iv)$, let $(X,Y)\sim \pi, (X^N, Y^N)\sim \hat{\pi}_c^N$ and recall that $$d_1(\hat{C}_N, C_\pi)= \mathcal{KR}((F_{X^N}, F_{Y^N})_{\#}\hat{\pi}_c^N, (F_X, F_Y)_{\#} \pi)$$ from Lemma \ref{lem:KR}. Lemma \ref{lem:general_marginals} yields
    \begin{align}\label{eq:easy} 
        \mathcal{KR}((F_{X^N}, F_{Y^N})_{\#}\hat{\pi}_c^N, (F_X, F_Y)_{\#} \pi) \le c\mathcal{KR}(\hat{\pi}_c^N, \pi) + \sup_{v\in [0,1]} |F_{Y^N}(v)-F_Y(v)|.
    \end{align}
    The first term in \eqref{eq:easy} converges to zero by  \eqref{asconv_AEM_c}. The second term converges to zero using \eqref{asconv_AEM_c}, \eqref{eq: general_inequalities} and Lemma \ref{lem:polya}.
\end{proof}

\begin{corollary}[Rates of convergence for \(\hat{C}_N\)]\label{thm_rates_C_N}
If $\pi$ has Lipschitz kernels \eqref{ass_Lip_kernel} and $y \mapsto F_Y(y)$ is Lipschitz continuous for $Y\sim \pi_2$, then 
\begin{align*}
\E[d_1(\hat{C}_N,C_\pi)] \le cN^{-1/3}.
\end{align*}
\end{corollary}
\begin{proof}
Recall from \eqref{eq_rate_conv_c} that $\E\left[\mathcal{KR}(\hat{\pi}^N_c,\pi)\right] \leq c N^{-1/3}$. Now let $(X,Y)\sim \pi, (X^N, Y^N)\sim \hat\pi^N_c$. From \eqref{eq:easy} we have
%\begin{align*}
%d_1(\hat{C}_N, C_{\pi})&=\int_{[0,1]^2} |\partial_1 C_{\pi}(u,v)-\partial_1 \hat{C}_N(u,v)|\,\de u \de v \\
%&=\int_{[0,1]^2} \big|F_{ Y| X= F_X^{-1}(u)}(v)-F_{Y^N|X^N=F_{X^N}^{-1}(u)}(F^{-1}_{Y^N}(v))\big|\,\de v \de u\\
%&\stackrel{\eqref{eq:easy}}{\le}  \mathcal{KR}(\hat\pi^N_c, \pi) +\sup_{v\in [0,1]} |F_{Y^N}(v)-v|\\
%&\stackrel{\eqref{eq_rate_conv_c}}{\le} cN^{-1/3} +\sup_{v\in [0,1]} |F_{Y^N}(v)-v|.
%\end{align*}
\begin{align*}
        d_1(\hat{C}_N, C_\pi)=\mathcal{KR}((F_{X^N}, F_{Y^N})_{\#}\hat{\pi}_c^N, (F_X, F_Y)_{\#} \pi) \le c\mathcal{KR}(\hat{\pi}_c^N, \pi) + \sup_{v\in [0,1]} |F_{Y^N}(v)-F_Y(v)|.
\end{align*}
Let $\widetilde Y^N\sim \varphi^N(Y)$. Applying now the Dvoretzky–Kiefer–Wolfowitz inequality, the Lipschitz property of $F_Y$ and \eqref{subcube_bound}, we conclude
\begin{align*}
    \E[d_1(C_{\pi}, \hat{C}_N)]&\le cN^{-1/3} +\E\Big[\sup_{v\in [0,1]} |F_{Y^N}(v)-F_Y(v)      |\Big]\\
    &\le cN^{-1/3} + \E\Big[\sup_{v\in [0,1]}|F_{Y^N}(v)-F_{\widetilde Y^N}(v)|\Big]  +  \sup_{v\in [0,1]}|F_{\widetilde Y^N}(v)-F_Y(v)| \\
    &\le cN^{-1/3}+ cN^{-1/2}+cN^{-1/3}\le c N^{-1/3},
\end{align*}
as claimed.
\end{proof}

\subsection{Comparison of $\hat{C}_N$ and checkerboard copula estimator} \label{sec:comp_checkerboard}
In this section, we assume $\pi \in \mathcal{P}_c(\R^2).$
In Theorem \ref{thm_conv_C_N} and Corollary \ref{thm_rates_C_N} we have established convergence and rates of the adapted copula estimator. The aim of this section is to compare this approach to the checkerboard copula estimator, and to find upper bounds for its convergence rate.

For this, let us first recall the setting of \cite[Section 3]{strothmann2022}. In particular, our aim is to define the checkerboard copula estimator $\hat{C}_N^{\#}$ with bandwidth $N_1=N_2=n$ (see \cite[(3.9)]{strothmann2022}) and relate it to an adapted empirical measure.

We proceed as follows: given the samples $(X_l,Y_l)_{l=1}^N$ it is natural to first define the $N$-checkerboard copula $\bar{C}_N:=C^{\#}[A_N]$, where
\begin{align*}
    a_{ij} = \begin{cases}
    1 & \exists\, l\in \{1,\dots, N\} \text{ such that } (\text{rank}(X_l), \text{rank}(Y_l))=(i,j),\\
    0 &\text{otherwise}.
    \end{cases}
\end{align*}
Formally, $\bar{C}_N$ can be obtained from the empirical copula $$\frac{1}{N}  \sum_{l=1}^N \mathbf{1}_{\big\{\frac{\text{rank}(X_l)}{N}\le u, \frac{\text{rank}(Y_l)}{N}\le v\big\}}$$
by uniformly distributing the mass at each point $(\text{rank}(X_l)/N, \text{rank}(Y_l)/N)$ on the cubes \(S_{ij}^N\).
As $n=N^{1/3}$ is assumed to be an integer, we obtain by \eqref{eq:CS_ij} that the $n$-checkerboard approximation of $\bar{C}_N$ is $\hat{C}^{\#}_N:=C^{\#}[A_n]$, an $n$-checkerboard copula, where \(A_n = (a_{ij})_{i,j=1}^n\) is given by
\begin{align}\label{eq:checkerboard_aij}
\begin{split}
    a_{ij} &= nV_{\bar C_N}(S_{ij}^n)\\
    &= n \frac{1}{N} |\{l\in \{1,\dots, N\}:\, (\text{rank}(X_l)/N, \text{rank}(Y_l)/N)\in S_{ij}^n \}|\\
    &=\frac{1}{n^2} |\{l\in\{1, \dots, n^3\}:\, (i-1)n^2< \text{rank}(X_l) \le in^2,\, (j-1)n^2< \text{rank}(Y_l) \le jn^2\}|.
\end{split}
\end{align}
In other words, the $n$-checkerboard approximation of $\bar{C}_N$ is the copula induced by the continuous adapted empirical measure of the observations
\begin{align*}
    \frac{(\text{rank}(X_l), \text{rank}(Y_l))}{N} =
    (F_{X^N}(X_l), F_{Y^N}(Y_l))  \qquad l=1, \dots, N,
\end{align*}
where $(X^N, Y^N)\sim \pi^N=1/N\sum_{l=1}^N \delta_{(X_l,Y_l)}.$ We call this measure the \emph{continuous empirical checkerboard measure}. It is formally defined as
\begin{align}\label{def_emp_checkerboard_measure}
    \hat{\pi}_c^{N,\#} &:= \sum_{i,j=1}^{n} \hat{\pi}^{N,\#}(S_{ij}^n) \cdot \cU^{S_{ij}^n},\\
    \nonumber\text{where} \qquad \hat{\pi}^{N,\#} &:= \frac 1 N \sum_{i=1}^N \delta_{(\varphi^N(\rank(X_i)/N),\varphi^N(\rank(Y_i)/N))}.
\end{align}
As \(n=N^{1/3}\) is assumed to be an integer, $\hat{\pi}_c^{N,\#}$ has uniform marginals and thus its distribution function is exactly $\hat{C}^{\#}_N$. Before deriving rates of convergence for the checkerboard estimator, we first recall strong consistency results for \(\hat{C}_N^\#\).

\begin{lemma}\label{lem_conv_C_N_check}
The checkerboard copula estimator \(\hat{C}_N^\#\) satisfies $\P$-almost surely:
    \begin{enumerate}[label = (\roman*)]
        \item \label{lem_conv_C_N_check1} \(\|\hat{C}_N^\# - C_\pi\|_\infty \to 0\),
        \item \label{lem_conv_C_N_check2} \(d_1(\hat{C}_N^\# ,C_\pi)\to 0\).
        \item \label{lem_conv_C_N_check3} \(\|(\hat{C}_N^\#)^\uparrow- C^\uparrow_\pi\|_\infty \to 0\),
        \item \label{lem_conv_C_N_check4} \(d_1((\hat{C}_N^\#)^\uparrow, C_\pi^\uparrow)\to 0\).
    \end{enumerate}
\end{lemma}

\begin{proof}
For the proof of \ref{lem_conv_C_N_check2}, see e.g.~\cite[Lemma 4.1]{Ansari-LFT-2023}. Then \ref{lem_conv_C_N_check4} follows from \ref{lem_conv_C_N_check2} using Corollary \ref{cor_incrRear}. Further, \ref{lem_conv_C_N_check2} implies \ref{lem_conv_C_N_check1} due to Figure \ref{fig:uu} and, similarly, \ref{lem_conv_C_N_check4} implies \ref{lem_conv_C_N_check3}.
\end{proof}

The following is a direct consequence of Lemma \ref{lem_conv_C_N_check} using the relations displayed in Figures \ref{fig:uu} and \ref{fig:uparrow}.

\begin{corollary}\label{cor:checkerboar}
    The continuous empirical checkerboard measure satisfies $\P$-almost surely: defining $\bar\pi=(F_X,F_Y)_{\#}\pi$ for $(X,Y)\sim \pi$ we have
    \begin{align*}
        \cIW(\hat{\pi}_c^{N,\#}, \bar\pi) &\to 0 , \quad &\cAW(\hat{\pi}_c^{N,\#}, \bar\pi) &\to 0, \quad \text{and} \quad &\mathcal{KR}(\hat{\pi}_c^{N,\#}, \bar\pi) &\to 0 \\
        \text{as well as} \quad \cIW((\hat{\pi}_c^{N,\#})^{\uparrow}, \bar\pi^\uparrow) &\to 0 , \quad &\cAW((\hat{\pi}_c^{N,\#})^\uparrow, \bar\pi^\uparrow) &\to 0, \quad \text{and} \quad &\mathcal{KR}((\hat{\pi}_c^{N,\#})^{\uparrow}, \bar\pi^\uparrow) &\to 0.
    \end{align*}
\end{corollary}

We now derive convergence rates for the checkerboard estimator $\hat{C}^{\#}_N$ and its associated coupling $\hat{\pi}_c^{N,\#}$ with respect to metrics that account for weak convergence of conditional distributions. To the best of our knowledge, these are the first such rates in the literature. In particular, our results resolve an open question of \cite[Section 5]{strothmann2022}.

\begin{theorem}\label{thm:rates}
    Defining $\bar \pi = (F_X,F_Y)_{\#}\pi$  for $(X,Y)\sim \pi\in \mathcal{P}_c(\R^2)$, assume that $\bar\pi$ has Lipschitz kernels. Then
    \begin{align}\label{eq:checkerboard_est}
        \E[\mathcal{CW}(\hat{\pi}_c^{N,\#}, \bar\pi)]\le \E[\mathcal{AW}(\hat{\pi}_c^{N,\#}, \bar\pi)]\le \E[\mathcal{KR}(\hat{\pi}_c^{N,\#}, \bar\pi)]\le cN^{-1/3}.
    \end{align}
    In particular,
    \begin{align*}
    \E[d_1(\hat{C}_N^{\#},C_\pi)] \le cN^{-1/3}.
    \end{align*}
\end{theorem}
%
%We remark that, contrary to Theorem \ref{thm_rates_C_N}, these rates are for $\hat{C}_N^{\#}$, which is not necessarily SI in general.
%
%
\begin{remark}
\begin{enumerate}[label = (\alph*)]\label{rem_ACE_CCE}
    \item Both the adapted copula estimator $\hat{C}_N$ and the checkerboard estimator 
$\hat{C}_N^\#$ are copulas by construction. Thus, they can be used as 
plug-in estimators for dependence measures, as we will see in  
Sections~\ref{sec_rank_based_DM} and~\ref{sec_RDM} below. Due to 
Theorem~\ref{thm_conv_C_N} and Lemma~\ref{lem_conv_C_N_check}, $\hat{C}_N$ and 
$\hat{C}_N^\#$ are strongly consistent estimators of $C_\pi$. Further, they 
exhibit the same rates of convergence; see Theorems~\ref{thm_rates_C_N} 
and~\ref{thm:rates}. Notably, both estimators can be computed in quasi-linear time, i.e. in $O(N \log N)$ operations, 
allowing them to be evaluated in a fraction of a second for a sample 
size $N$ larger than a million.
    \item  Both $\hat{C}_N$ and $\hat{C}_N^\#$ are based on binning methods. For the adapted copula estimator, 
all $N$ observations are first binned into the subcubes of $[0,1]^2$, then mass is uniformly distributed within each subcube, and the 
resulting measure is subsequently transformed to \(\cU\)-marginals; see Algorithm \ref{alg:ACE}. This 
approach offers the advantage of requiring only a single parameter that 
determines the edge length of the subcube; typically, the edge length is 
chosen as $n = (\lfloor N^{1/3} \rfloor)^{-1}$; see \cite[Definition 1.2]{Backhoff-Bartl-Beiglboeck-Wiesel-2022}. In contrast, the checkerboard 
estimator first transforms the data to uniform marginals by considering standardized rank-transformed observations. These transformed observations 
are then distributed via binning onto an $n \times n$ checkerboard, where the 
mass within each subcube is uniformly distributed. To ensure that the 
resulting distribution has uniform marginals on $[0,1]$---and thus constitutes 
a valid copula---we have assumed that $n = N^{1/3}$ is an integer, so that 
every column and row contains exactly $n^2$ observations; see e.g.~Figure \ref{fig_comparison_AEC_checkCop}. In the general case, a good performance of the checkerboard estimator requires 
additional bandwidth parameters to partition $[0,1]^2$ into suitable rectangular 
bins; see \cite[Section~4]{strothmann2022}.
\item \label{rem_ACE_CCE2} While $\hat{C}_N$ and $\hat{C}_N^\#$ exhibit similar properties, an advantage of the adapted copula estimator lies in the fact that no additional bandwidth parameters are necessary.
An advantage of the checkerboard estimator is that it can directly be applied to samples drawn from general continuous distribution functions. In contrast, the adapted copula estimator works for any marginal transformation to a distribution with domain in \([0,1]^2\), e.g. for \((f,g)_\#\pi\) with \(f,g\colon \R\to (0,1)\) continuous and strictly increasing. 
\item The adapted copula estimator extends to more general settings as follows:
\begin{itemize}
    \item Our consistency results and rates for $\hat{\pi}^N$ and $\hat{\pi}_c^N$ hold for general $\pi \in \mathcal{P}_c(A\times B)$, where $A,B\subseteq \R$ are bounded sets and the sample complexity scales with $\mathrm{diam}(A)\vee \mathrm{diam}(B)$.
    \item One can also consider extensions to multivariate distributions $\pi \in \mathcal{P}([0,1]^k \times [0,1]^l).$ In this case, we still expect consistency to hold, while our rates are conjectured to be of order $N^{-1/(k+l)}$ for $k,l\ge 2$; see \cite{Backhoff-Bartl-Beiglboeck-Wiesel-2022} for the case $k=l=d.$
\end{itemize}
\end{enumerate}
\end{remark}

\begin{proof}[Proof of Theorem \ref{thm:rates}]
Recall that \(\{(X_l, Y_l)\}_{l=1}^N\) are i.i.d.~and let $(X^N, Y^N)\sim 1/N\sum \delta_{(X_l,Y_l)}.$ We first recall the Dvoretzky–Kiefer–Wolfowitz inequality, stating that
    \begin{align}\label{eq:DKW}
        \E\Big[\sup_{y\in \R} |F_{ Y^N}(y)-F_Y(y)|\Big]\le c N^{-1/2}.
    \end{align}
    To simplify notation, let us write
    $$ (U_l,V_l):=(F_{X^N}(X_l), F_{Y^N}(Y_l)), \qquad l=1, \dots, N.$$
    Using the same argument as in the proof of Theorem \ref{lem:triangle_est} with the samples $(X_l,Y_l)_{l=1}^N$ replaced by $(U_l,V_l)_{l=1}^N$, we have
    \begin{align*}
        \mathcal{KR}(\hat{\pi}^{N,\#}_c,\bar \pi)\le \mathcal{KR}(\hat{\pi}^{N, \#},\bar \pi) + cN^{-1/3},
    \end{align*}
    where $\hat{\pi}^{N, \#}$ denotes the (discontinuous) adapted empirical measure constructed from $(U_l,V_l)_{l=1}^N.$
    It is thus sufficient to estimate $\mathcal{KR}(\hat{\pi}^{N, \#},\bar\pi).$
    For this, note that the sample $(F_X(X_l),F_Y(Y_l))_{l=1}^N$ is different from $(U_l,V_l)_{l=1}^N.$ In particular, $(U_l,V_l)_{l=1}^N$ are not i.i.d.~(e.g.~as the sum $\sum U_l=(N+1)/2$ is almost surely deterministic). The proof will proceed by carefully arguing that $(U_l,V_l)_{l=1}^N$ can be replaced by 
    $$(\bar U_l, \bar V_l) :=(F_X(X_l),F_Y(Y_l))\qquad l=1, \dots, N$$ with negligible error. 
    For this, let us first note that $(\bar U_l, \bar V_l)_{l=1}^N$ are i.i.d.~with distribution $\overline{\pi}$. For notational simplicity, we now assume that $X_1 \le \dots \le X_N$. We then note that, by definition, exactly the $m:=N/n = N^{2/3}$ observations $(U_{(i-1)m+1}, \dots, U_{im})$ fall into each $S_i^n = ((i-1)/n, i/n]$, $i=1,\dots, n$, and thus, using the disintegration theorem,
    \begin{align*}
        \hat{\pi}^{N,\#} = \int_0^1 (\hat{\pi}^{N,\#})_{u}\, \hat{\pi}^{N,\#}_1(\de u)= \frac{1}{n} \sum_{i=1}^n \delta_{\varphi^N(S_i^n)} \otimes (\hat{\pi}^{N,\#})_{u},
    \end{align*}
    where for $u\in S_i^n$
    \begin{align*}
        (\hat{\pi}^{N,\#})_{u}:= \frac{1}{m} \sum_{j=(i-1)m+1}^{im} \delta_{\varphi^N(V_j)}.
    \end{align*}
    Next we compute
    \begin{align*}
        \mathcal{KR}(\hat{\pi}^{N, \#}, \bar \pi) = \int_0^1 |\varphi^N(u)-u|\,du + \sum_{i=1}^n \int_{S_i^n} \mathcal{W}( (\hat{\pi}^{N, \#})_u,\bar \pi_u)\,\de u.
    \end{align*}
    It is thus sufficient to bound $\mathcal{W}( (\hat{\pi}^{N, \#})_u,\bar \pi_u).$ For this we use the triangle inequality to obtain 
    \begin{align*}
        \mathcal{W}((\hat{\pi}^{N, \#})_u, \bar\pi_u) &\le  \mathcal{W}\Bigg(\bar \pi_u, \frac{1}{m} \sum_{j=(i-1)m+1}^{im}  \bar\pi_{\bar U_j}\Bigg) + \mathcal{W} \Bigg( \frac{1}{m} \sum_{j=(i-1)m+1}^{im} \bar\pi_{\bar U_j}, \frac{1}{m} \sum_{j=(i-1)m+1}^{im} \delta_{\varphi^N(\bar V_j)} \Bigg)\\
        &\quad +\mathcal{W}\Bigg( \frac{1}{m} \sum_{j=(i-1)m+1}^{im} \delta_{\varphi^N(\bar V_j)}, \frac{1}{m} \sum_{j=(i-1)m+1}^{im} \delta_{\varphi^N(V_j)} \Bigg).
    \end{align*}
    We now estimate each term separately. For this, 
    by convexity of $\mathcal{W}$, as $\bar\pi$ has Lipschitz kernels, we have
    \begin{align*}
        \mathcal{W}\Bigg(\bar \pi_u, \frac{1}{m} \sum_{j=(i-1)m+1}^{im} \bar\pi_{\bar U_j}\Bigg)
        &\le  \frac{c}{m} \sum_{j=(i-1)m+1}^{im} |u-\bar U_j| \\
        &\le \frac{c}{n}+ c\,n \int_{(i-1)/n}^{i/n} |v- F_{\bar U^N}^{-1}(v)|\,\de v, 
    \end{align*}
    where $\bar U^N\sim 1/N \sum_j \delta_{\bar U_j}$, so that 
    \begin{align*}
        \sum_{i=1}^n \int_{S_i^n} \mathcal{W}\Bigg(\bar\pi_u, \frac{1}{m} \sum_{j=(i-1)m+1}^{im} \bar\pi_{\bar U_j}\Bigg) \,\de u\le \frac{c}{n} + c\,\mathcal{W}(\mathcal{U}, F_X{}_{\#}\pi^N_1).
    \end{align*}
    Next we note that, given $(\bar U_{(i-1)m+1} , \dots, \bar U_{im})$, the random variables $(\bar V_{(i-1)m+1}, \dots, \bar V_{im})$ are independent with laws $(\bar \pi_{\bar U_{(i-1)m+1}}, \dots, \bar \pi_{\bar U_{im}})$. Using the variance bound 
    $\E[\cW(\mu, \nu)]\le \int_0^1 \sqrt{\text{Var}(F_\mu(t)-F_{\nu}(t))}\,\de t$ we thus find
    \begin{align*}
        \E\Bigg[\mathcal{W} \Bigg( \frac{1}{m} \sum_{j=(i-1)m+1}^{im} \bar\pi_{\bar U_j}, \frac{1}{m} \sum_{j=(i-1)m+1}^{im} \delta_{\varphi^N(\bar V_j)} \Bigg)\Bigg|\bar U_{(i-1)m+1} , \dots, \bar U_{im}\Bigg] \le  \frac{c}{\sqrt{m}} +\frac{1}{n} \le cN^{-1/3}.
    \end{align*}
    Lastly, 
    \begin{align*}
        \mathcal{W}\Bigg( \frac{1}{m} \sum_{j=(i-1)m+1}^{im} \delta_{\varphi^N(\bar V_j)}, \frac{1}{m} \sum_{j=(i-1)m+1}^{im} \delta_{\varphi^N(V_j)} \Bigg) 
        &\le  \frac{1}{m} \sum_{j=(i-1)m+1}^{im} |\varphi^N(\bar V_j)-\bar V_j|+|\bar V_j-V_j|+|V_j-\varphi^N(V_j)|\\
        &\stackrel{\eqref{subcube_bound}}{\le} \frac{c}{n} + \sup_{y\in \R}|F_{Y^N}(y)-F_Y(y)| + \frac c n.
    \end{align*}
    Combining all the above estimates with \eqref{eq:DKW}, we obtain
    \begin{align*}
        \E[\mathcal{KR}(\hat{\pi}^{N, \#}, \bar\pi)]\le cN^{-1/3}+cN^{-1/2}\le cN^{-1/3}
    \end{align*}
    as claimed. Recalling that $\hat{\pi}^{N, \#}_c, \bar \pi\in \Pi(\mathcal{U}, \mathcal{U})$ and using Lemma \ref{lem:KR}, the second claim follows.
\end{proof}

While similar in spirit, the estimators $\hat{C}_N$ and $\hat C^{\#}_N$ are clearly different.
Our adapted empirical copula estimator \(\hat{C}_N\) places the observations directly into the cubes $S_{ij}^n$, whereas the checkerboard estimator \(\hat{C}^\#_N\) distributes rank-transformed observations over the cubes; see the following example.

\begin{figure}[h!]
    \centering
    \begin{align*}
    (\hat{\pi}^N_c(S_{ij}^n))_{i,j=1}^n &= \frac{1}{n^3}
    \begin{pmatrix}
        n^3 & 0 & \cdots & 0\\
        0 & 0 & \cdots & 0 \\
        \vdots & \vdots & \ddots & \vdots \\
        0 & 0 & \cdots & 0
    \end{pmatrix}
    \quad &&\text{and}& \quad
    (\hat{\pi}^{N,\#}_c(S_{ij}^n))_{i,j=1}^n &= \frac{1}{n^3}
    \begin{pmatrix}
        n^2 & 0 & \cdots & 0 \\
        0 & n^2 & \cdots & 0 \\
        \vdots & \vdots & \ddots & \vdots \\
        0 & 0 & \cdots & n^2
    \end{pmatrix}\\
    A_n &= \frac{1}{n^2}
    \begin{pmatrix}
        n & n & \cdots & n\\
        n & n & \cdots & n \\
        \vdots & \vdots & \ddots & \vdots \\
        n & n & \cdots & n
    \end{pmatrix}
    \quad &&\text{and} &\hspace{1.6cm} 
    A_n^{\#} &= \frac{1}{n^2}
    \begin{pmatrix}
        n^2 & 0 & \cdots & 0 \\
        0 & n^2 & \cdots & 0 \\
        \vdots & \vdots & \ddots & \vdots \\
        0 & 0 & \cdots & n^2
    \end{pmatrix}
    \end{align*}
    \caption{Representation of estimated distributions (top row) and their associated checkerboard copulas (bottom row) for Example \ref{ex:checkerboard}. The left column corresponds to the adapted empirical copula estimator, the right column to the checkerboard estimator. By convention, each row/column of the doubly stochastic matrices \(A_n\) and \(A_n^\#\) sums up to \(1\), whereas the densities of \(\hat{\pi}_c^N\) and \(\hat{\pi}^{N,\#}_c\) integrate to \(1\).}
    \label{fig_comparison_AEC_checkCop}
\end{figure}

\begin{example}[\(\hat{C}_N \ne \hat{C}_N^\#\)]\label{ex:checkerboard}
For \(\pi\in \Pi(\cU,\cU)\), assume that $X_1, \dots, X_N \in (0,1/n]$ and $Y_1, \dots, Y_N\in (0,1/n]$ as well as \(X_1 < X_2 < \cdots < X_N\) and \(Y_1 < Y_2 < \cdots < Y_N\).  Since all observations fall into the subcube \(S^n_{11} = (0,1/n]\times (0,1/n]\), we have $\hat{\pi}^N_c=\mathcal{U}^{(0,1/n]\times(0,1/n]}.$ Hence, the adapted empirical copula estimator is the independence copula, i.e. 
\begin{align}\label{ex:comp_check_adap_est1}
    \hat{C}_N(u,v) = uv = C^\perp(u,v).
\end{align}
In contrast, the checkerboard estimator $\hat{C}^\#_N$ is constructed via the ranks of the observations. Since \(\text{rank}(X_i) = i = \text{rank}(Y_i)\) for all \(i\in \{1,\ldots,N\}\), the doubly stochastic matrix \(A_n^\#\) constructed via \eqref{eq:checkerboard_aij} is diagonal with \(a_{ij} = 1\) if \(i=j\) and \(a_{ij} = 0\) otherwise; see Figure \ref{fig_comparison_AEC_checkCop}.
Then, a straightforward calculation shows that the associated checkerboard copula defined by \eqref{def:check_cop} is given as
\begin{align}\label{ex:comp_check_adap_est2} 
\begin{split}
    \hat{C}^\#_N(u,v) 
    %&= \begin{cases}
     %   \frac{k-1}{n} + (\frac{k}{n}-\frac{k-1}{n})\, C^\perp\left(\frac{u-\frac{k-1}{n}}{\frac{k}{n}-\frac{k-1}{n}},\frac{v-\frac{k-1}{n}}{\frac{k}{n}-\frac{k-1}{n}}\right), &\text{if } (u,v)\in S_{k,k},\\
    %    \min\{u,v\}, &\text{else}.
    %\end{cases} \\
    &= \begin{cases}
        \frac{k-1}{n} + n (u-\frac{k-1}{n})(v-\frac{k-1}{n}), &\text{if } (u,v)\in S_{k,k}^n,\\
        \min\{u,v\}, &\text{else}.
    \end{cases}
\end{split}
\end{align}
Obviously, for \(N\geq 8\) (i.e.~\(n = N^{1/3} > 1\)), the estimators in \eqref{ex:comp_check_adap_est1} and \eqref{ex:comp_check_adap_est2} differ. Note that \(\hat{C}_N\) and \(\hat{C}_N^\#\) are both SI in this example and thus \(\hat{C}_N = (\hat{C}_N)^\uparrow\) as well as \(\hat{C}_N^\# = (\hat{C}_N^\#)^\uparrow\).
%one  since \(\hat{A}_n\) is a diagonal matrix, the checkerboard estimator is given by a so-called ordinal sum of the independence copula \(C^\perp\) with respect to the intervals \(((\tfrac{k-1}{n},\tfrac{k}{n}])_{1\leq k \leq n}\) (see e.g. \cite[Definition 3.1.8]{fdsempi2016}), that is, 
\end{example}

\section{Rank-based dependence measures}\label{sec_rank_based_DM}

Certainly the most prominent dependence measure that satisfies Axioms \ref{axiom1}--\ref{axiom3} is the population version of \emph{Chatterjee's rank correlation} \(\xi\). It is defined for $\pi \in \mathcal{P}(\R^k \times \R)$ and \((X,Y) \sim \pi\) by
\begin{align}\label{def_xi}
\begin{split}
    \xi(\pi) := &\frac{\int |\P(Y\geq y\mid X = x) - \P(Y\geq y) |^2 \, \pi_2(\de y) \pi_1(\de x)}{\int  |\mathbf{1}_{\{z\geq y\}} - \P(Y\geq y) |^2 \, \pi_2(\de z) \pi_2(\de y) }; \\
    %= & \frac{\int \int \left(\bar{F}_{\pi_x} (F_{\nu}^{-1}(v)) - \bar{F}_{\nu}(F_{\nu}(v)\right)^2 \de v \de \mu(x)}{\int_{\R} \Var(\1_{\{Y \geq y\}}) \de \nu(y)} =: \xi(\pi);
\end{split}
    \end{align}
see \cite{chatterjee2021,gamboa2020} and the bivariate versions in \cite{chatterjee2020,siburg2013}. 
%Here, \(\bar{F}_{\pi_x}(y)=P(Y\geq y\mid X=x)\) and \(\bar{F}_\nu(y) = P(Y\geq y)\) denote the survival function of \(\pi_x\) and \(\nu\), respectively.
The quantity \(\xi\) has attracted a lot of attention in recent years. It admits a simple and elegant nearest neighbor-based estimator with asymptotic theory \cite{dette2023boot,han2022limit,lin2023boosting,shi2022power,shi2024azadkia} and it has important applications such as model-free feed forward feature selection and variable selection; see \cite{ansari2023MFOCI,chatterjee2021,deb2020b}. 

In this section, we consider dependence measures that are constructed via conditional survival functions, extending Chatterjee's \(\xi\) in \eqref{def_xi}. As above, we consider laws $\pi\in\mathcal{P}(\R^{k}\times \R)$, i.e.~if $(X,Y)\sim \pi$, then $Y$ is a one-dimensional random variable.
For a function \(\varphi:\R\to \R\), that is convex and strictly convex at \(0\) with \(\varphi(0) = 0\), we consider the functional
\begin{align}\label{def:T_phi}
\begin{split}
\xi_\varphi(\pi) :=   &\frac{\int  \varphi\big( \P(Y\geq y \mid X=x) - \P(Y\geq y)\big) \, \pi_2(\de y)\,\pi_1(\de x)}{\int \varphi\big( \mathbf{1}_{\{z\ge y\}}- \P(Y\geq y)\big) \,\pi_2(\de z)\,\pi_2(\de y)} \\
=&\frac{\int \int_0^1 \varphi\big( \bar{F}_{Y|X=x}(F_Y^{-1}(u))-\bar{F}_Y(F_Y^{-1}(u))\big) \,\de u\,\pi_1(\de x)}{\int \varphi\big( \mathbf{1}_{\{z\ge y\}}- \bar{F}_Y(y)\big) \,\pi_2(\de z)\,\pi_2(\de y)}.
\end{split}
\end{align}
Here, \(z\mapsto \bar{F}_Z(z) := P(Z\geq z)\) denotes the survival function of a random variable \(Z\). Obviously, \(\bar{F}(z) = 1- F_Z(z) + P(Z=z)\) and, if \(F_Z\) is continuous then \(\bar{F}_Z(z) = 1 - F_Z(z)\) for all \(z\).
Since \(\pi\) is assumed to be non-degenerate, the denominator of \(\xi_\varphi(\pi)\) is positive so that $\xi_\varphi(\pi)$ is well-defined.
The functional $\xi_\varphi$ is a dependence measure that satisfies Axioms \ref{axiom1}--\ref{axiom3}, see \cite[Theorem 4.1]{Ansari-Fuchs-2026} for a variant based on conditional distribution functions.  For \(\varphi(x) = x^2\),
\(\xi_\varphi\) coincides with Chatterjee's rank correlation \(\xi\) in \eqref{def_xi}.
We also remark that, for continuous \(y\mapsto F_Y(y)\), the denominator of \(\xi_\varphi(\pi)\) reduces to
\begin{align}\nonumber
    \int \varphi\left(\1_{\{z\geq y\}} - \bar{F}_Y(y)\right) \, \pi_2(\de z) \, \pi_2(\de y) &= \int \varphi\big(1 - \bar{F}_Y(y)\big) \bar{F}_Y(y) + \varphi\big(0 - \bar{F}_Y(y)\big) F_Y(y) \, \pi_2(\de y) \\
    \label{eq_denom_xi_phi}&= \int_0^1 \varphi(u)\, (1-u) + \varphi(u-1)\,u \, \de u.
\end{align}
Recalling that $\varphi$ is convex, its restriction $\varphi:[-1,1]\to \R$ is $L$-Lipschitz. We will make use of this fact throughout this section. 

The following result is an application of the triangle inequality and shows that $\xi_\varphi$ is naturally controlled by $\mathcal{CW}.$ For notational simplicity we define 
\begin{align*}
    D(\pi, \tilde\pi)&:= \int \int_0^1 \varphi\big(\bar{F}_{Y|X=x}(F_Y^{-1}(u))-\bar{F}_Y(F_Y^{-1}(u))\big)\,\de u\,\pi_1(\de x) \\
    &\quad -\int \int_0^1 \varphi\big({\bar{F}}_{\tilde Y|\tilde X=x}(F_{\tilde Y}^{-1}(u))-\bar{F}_{\tilde Y}(F_{\tilde Y}^{-1}(u))\big)\,\de u\,\tilde\pi_1(\de x)
\end{align*}
for $\pi, \tilde\pi\in \mathcal{P}(\R^k\times \R)$ and $(X,Y)\sim \pi, (\tilde X, \tilde Y)\sim \tilde \pi$.

\begin{lemma}\label{lem:estimate}
We have
\begin{align*}
| D(\pi, \tilde\pi)| &\le L \big[\mathcal{CW} \big((x,F_Y)_{\#}\pi, (x,F_{\widetilde Y})_{\#}\tilde\pi\big) + \mathcal{W}\big(F_Y{}_\#\pi_2,F_{\widetilde Y}{}_\#\tilde\pi_2\big)\big]\\
&\le 2L \cdot \mathcal{CW} \big((x,F_Y)_{\#}\pi, (x,F_{\widetilde Y})_{\#}\tilde\pi \big).
\end{align*}
\end{lemma}

\begin{proof}
We first note that for any $x,\tilde x\in \R^k,$
\begin{align*}
&\int_{0}^1 \varphi\big(\bar{F}_{Y|X=x}(F_Y^{-1}(u))-\bar{F}_Y(F_Y^{-1}(u))\big)\,\de u-\int_{0}^1 \varphi\big(\bar{F}_{\widetilde Y|\widetilde X=\tilde x}(F_{\widetilde Y}^{-1}(u))-\bar{F}_{\widetilde Y}(F_{\widetilde Y}^{-1}(u))\big)\,\de u \\
&\le L \int_{0}^1 \Big|\bar{F}_{Y|X=x}(F_Y^{-1}(u)) - \bar{F}_Y(F_Y^{-1}(u)) + \bar{F}_{\widetilde Y}(F_{\widetilde Y}^{-1}(u)) - \bar{F}_{\widetilde Y|\widetilde X=\tilde x}(F_{\widetilde Y}^{-1}(u))\Big| \,\de u \\
&\le  L \big[\mathcal{W} ({F_{Y}}_{\#} \pi_{x}, F{_{\widetilde Y}}_{\#} \tilde\pi_{\tilde x}) + \cW(F_Y{}_\#\pi_2,F_{\widetilde Y}{}_\#\tilde\pi_2)\big],
\end{align*}
where we use for the last equality that \(F(y)\geq u\) if and only if \(y\geq F^{-1}(u)\) for all \(u\in (0,1)\), \(y\in \R\), and for any distribution function \(F\).
Thus, for any $\gamma\in \Pi(\pi_1, \tilde\pi_1)$ we have 
\begin{align*}
|D(\pi, \tilde\pi)| \le L \bigg[\int \mathcal{W} ({F_{Y}}_{\#} \pi_x, F{_{\widetilde Y}}_{\#} \tilde\pi_{\tilde x})\,\gamma(\de x, \de \tilde x) + \cW(F_Y{}_\#\pi_2,F_{\widetilde Y}{}_\#\tilde\pi_2)\bigg].
\end{align*}
Taking the infimum over $\gamma \in \Pi(\pi_1, \tilde\pi_1)$ and recalling Lemma \ref{lem:second} concludes the proof.
\end{proof}

A similar result can be obtained for \emph{measures of sensitivity} \cite{Ansari-LFT-2023} of the form
\begin{align}
    \Lambda_{\varphi}(\pi) = \frac{\int \int_{0}^1 \varphi\big(\bar{F}_{Y|X=x}(F_Y^{-1}(u))-\bar{F}_{Y|X=x'}(F_Y^{-1}(u))\big)\,\de u\,\pi_1(\de x)\, \pi_1(\de x')}{\int  \varphi\big( \mathbf{1}_{\{z\ge y\}}-\mathbf{1}_{\{z'\ge y\}}\big)\,\pi_2(\de y)\,\pi_2(\de z)\,\pi_2(\de z')},
\end{align}
where \(\Lambda_\varphi\) satisfies Axioms \ref{axiom1}--\ref{axiom3} whenever \(\varphi\colon \R\to \R\) is convex and strictly convex at \(0\) with \(\varphi(0) = 0\);  see \cite{Ansari-Fuchs-2026} for a variant based on conditional distribution functions. 
In fact, following the same arguments as in the proof of Lemma \ref{lem:estimate}, we have the following: 

\begin{corollary}\label{lem:estimate2}
For $\pi, \tilde\pi\in \mathcal{P}(\R^k\times \R)$ and $(X,Y)\sim \pi, (\tilde X, \tilde Y)\sim \tilde \pi$ we have
\begin{align*}
&\bigg|\int  \varphi\big(\bar{F}_{Y|X=x}(F_Y^{-1}(u))-\bar{F}_{Y|X=x'}(F_Y^{-1}(u))\big)\,\de u\,\pi_1(\de x)\, \pi_1(\de x') \\
&\qquad-\int  \varphi\big(\bar{F}_{\widetilde Y|\widetilde X=\tilde x}(F_{\widetilde Y}^{-1}(u))-\bar{F}_{\widetilde Y|\widetilde X=\tilde x'}(F_{\widetilde Y}^{-1}(u))\big)\,\de u\,\tilde\pi_1(\de \tilde x)\, \tilde\pi_1(\de \tilde{x}') \bigg| \\
&\qquad\qquad\le 2L\cdot \mathcal{CW} \big((x,F_Y)_{\#}\pi, (x,F_{\widetilde Y}) _{\#}\tilde\pi \big).
\end{align*}
\end{corollary}
Lemma \ref{lem:estimate} and Corollary \ref{lem:estimate2} imply the following.
\begin{theorem}[\(\cIW\)-continuity of \(\xi_\varphi\) and \(\Lambda_\varphi\)]\label{thm:continuity}
Let $(\pi^n)$ be a sequence of probability measures in $\mathcal{P}(\R^k\times \R)$ and let $\pi\in \mathcal{P}(\R^k\times \R)$. Take $(X^n,Y^n)\sim \pi^n$, $(X,Y)\sim \pi$ and define $\tilde\pi^n:= (x,F_{Y^n})_\#\pi^n$ and $\tilde\pi:= (x,F_Y)_\#\pi$. Then we have
\begin{align*}
\mathcal{CW}(\tilde\pi^n, \tilde\pi)\to 0 \quad \Rightarrow\quad  \xi_\varphi(\pi^n) \to \xi_\varphi(\pi) \quad and \quad \Lambda_\varphi(\pi^n) \to \Lambda_\varphi(\pi).
\end{align*}
\end{theorem}

\begin{proof}
    By Lemma \ref{lem:estimate}, the numerator of \(\xi_\varphi(\pi^n)\) converges to the numerator of \(\xi_\varphi(\pi)\). For the denominators, we first note that 
    \begin{align*}
        \int \varphi\big(\mathbf{1}_{\{z\ge y\}} -\bar{F}_{Y^n}(y)\big)\, \pi^n_2(dz)= \bar{F}_{Y^n}(y)\, \varphi(1-\bar{F}_{Y^n}(y)) + (1- \bar{F}_{Y^n}(y)) \,\varphi(-\bar{F}_{Y^n}(y)) =: f(\bar{F}_{Y^n}(y)).
    \end{align*}
    As $f$ is Lipschitz, the claim follows from $\mathcal{W}(F_{Y^n}{}_{\#}\pi_2^n, F_Y{}_{\#}\pi_2)\to 0$. 
    The statement for \(\Lambda_\varphi\) follows similarly with Corollary \ref{lem:estimate2} noting that the denominator 
    \begin{align*}
        &\int \varphi\left( \1_{\{z\geq y\}} - \1_{\{z'\geq y\}}\right)\, \pi^n_2(\de z) \,\pi^n_2 (\de z') \\
        & = \left[\varphi(-1) + \varphi(1)\right] \bar{F}_{Y^n}(y)(1-\bar{F}_{Y^n}(y)) + \varphi(0) [\bar{F}_{Y_n}(y)^2+(1-\bar{F}_{Y_n}(y))^2]
    \end{align*}
    is Lipschitz in \(\bar{F}_{Y^n}(y)\).
\end{proof}
If the cdf \(y\mapsto F_Y(y)\) is continuous, then  $
\mathcal{CW}(\pi,\pi^n)\to 0$ implies \(\mathcal{CW}(\tilde\pi^n,\tilde\pi)\to 0\), as the following result states.
\begin{corollary} \label{cor:rearr}
In the setting of Theorem \ref{thm:continuity}, let $y\mapsto F_Y(y)$ be continuous. Then we have
\begin{align*}
\mathcal{CW}(\pi^n, \pi)\to 0 \quad \Rightarrow\quad \mathcal{CW}(\tilde{\pi}^n, \tilde{\pi})\to 0 \quad \Rightarrow\quad  \xi_\varphi(\pi^n) \to \xi_\varphi(\pi) \quad \text{and} \quad \Lambda_\varphi(\pi^n) \to \Lambda_\varphi(\pi).
\end{align*}
\end{corollary}

\begin{proof}
First recall from Lemma \ref{lem:second} that \(\cIW(\pi^n,\pi)\to 0\) implies \(\cW(\pi^n_2,\pi_2)\to 0\). Then, continuity of \(y\mapsto F_Y(y)\) and Polya's lemma together with $\mathcal{W}(\pi_2^n, \pi_2)\to 0$ imply $\sup_{y\in \R}|F_{Y^n}(y)-F_Y(y)|\to 0$. Next, by the triangle inequality we note that
\begin{align}\label{eq:IW1}
\begin{split}
\mathcal{CW}\big((x,F_{Y^n}){}_\#\pi^n, (x,F_{Y}){}_\#\pi\big)
&\le \mathcal{CW}\big( (x,F_{Y^n})_{\#}\pi^n,(x,F_{Y}){}_\#\pi^n\big) 
+ \mathcal{CW} \big((x,F_{Y}){}_\#\pi^n, (x,F_{Y}){}_\#\pi\big)\\
&\le \int \mathcal{W}(F_{Y^n}{}_\#\pi^n_{x}, F_{Y}{}_\# \pi^n_{x})\,\pi^n_1(\de x)
+ \mathcal{CW} \big((x,F_{Y}){}_\#\pi^n, (x,F_{Y}{})_\#\pi\big)\\
&\le \sup_{y\in \R} |F_{Y^n}(y)-F_Y(y)|  + \mathcal{CW} \big((x,F_{Y}){}_\#\pi^n, (x,F_{Y}){}_\#\pi\big).
\end{split}
\end{align}
The first term converges to zero, as shown above. For the second term we note that
$\mathcal{CW} (\pi^n,\pi)\to 0$ by assumption, so that the claim follows for Lipschitz continuous $F_Y$, noting that $((x,F_{Y})_{\#}\pi^n)_x=F_{Y}{}_\#\pi^n_x$. The general case follows from a standard approximation result of continuous functions by Lipschitz functions, as $y \mapsto F_Y(y)$ is uniformly continuous.
\end{proof}

Consider now the estimators
\begin{align}
    \hat{\xi}_\varphi^N := \xi_\varphi(\hat{\pi}^{N}_c) \quad \text{and} \quad \hat{\Lambda}_\varphi^N := \Lambda_\varphi(\hat{\pi}^N_c)
\end{align}
based on the continuous adapted empirical measure $\hat{\pi}^N_c$, as well as the checkerboard estimators
\begin{align}
    \hat{\xi}_\varphi^{N,\#} := \xi_\varphi(\hat{\pi}^{N,\#}) \quad and \quad \hat{\Lambda}_\varphi^{N,\#} := \Lambda_\varphi(\hat{\pi}^{N,\#}).
\end{align}
The following result follows from Theorem \ref{thm:continuity} together with \eqref{asconv_AEM} and \eqref{eq: general_inequalities}.

\begin{corollary}[Strong consistency]\label{cor:strong_consis}
Take $(X,Y)\sim \pi\in \mathcal{P}_c([0,1]^2)$. Then we have \(\hat{\xi}_\varphi^N \to \xi_\varphi(\pi)\) and \(\hat{\Lambda}_\varphi^N \to \Lambda_\varphi(\pi)\), as well as \(\hat{\xi}_\varphi^{N,\#} \to \xi_\varphi(\pi)\) and \(\hat{\Lambda}_\varphi^{N,\#} \to \Lambda_\varphi(\pi)\) \(\P\)-almost surely as \(N\to \infty\).
\end{corollary}
\begin{proof}
    The claim then follows from Corollary \ref{cor:rearr} recalling  $\mathcal{CW}( \hat{\pi}^N_c, \pi)\to 0$ as a consequence of Theorem \ref{lem:triangle_est} and \eqref{eq: general_inequalities}.
\end{proof}

Finally, we obtain the following rates for plug-in estimators of dependence measures based on Section \ref{sec:estimation}.

\begin{theorem}[Rates of convergence]\label{cor:rates}
Assume that $\pi\in \mathcal{P}_c([0,1]^2)$ has Lipschitz kernels and that $y\mapsto F_Y(y)$ is Lipschitz. Then we have
\begin{align*}
    \E\big[|\hat{\xi}_\varphi^N - \xi_\varphi(\pi)|\big] &\le cN^{-1/3} \quad &&\text{and} \quad &\E\big[|\hat{\Lambda}_\varphi^N - \Lambda_\varphi(\pi)|\big] &\le cN^{-1/3} \\
    \text{as well as} \qquad \E\big[|\hat{\xi}_\varphi^{N,\#} - \xi_\varphi(\pi)|\big] &\le cN^{-1/3} \quad &&\text{and} \quad &\E\big[|\hat{\Lambda}_\varphi^{N,\#} - \Lambda_\varphi(\pi)|\big] &\le cN^{-1/3}.
\end{align*}
\end{theorem}

\begin{proof}
Let $(X,Y)\sim \pi$ and $(X^N,Y^N)\sim \hat{\pi}^N_c$. For the numerator of \(\xi_\varphi\),
    note that Lemma \ref{lem:estimate} yields
    \begin{align*}
        |D(\pi, \hat{\pi}^N_c)| \le 2L \cdot \mathcal{CW} \big((x,F_{Y^N}(y))_{\#}\hat{\pi}^N_c ,(x,F_Y(y))_{\#}\pi\big).
    \end{align*}
    Next, recall that $\pi \in \mathcal{P}_c([0,1]^2)$ and $F_Y$ is Lipschitz, so \eqref{eq:IW1} gives
    \begin{align*}
        \mathcal{CW} \big((x,F_{Y^N}(y))_{\#}\hat{\pi}^N_c, (x,F_Y(y))_{\#}\pi \big) &\le \sup_{y\in \R} |F_{Y^N}(y)-F_Y(y)| +\mathcal{CW}((x,F_Y)_{\#}\hat{\pi}^N_c, (x,F_Y)_{\#}\pi)\\
        &\le \sup_{y\in \R} |F_{Y^N}(y)-F_Y(y)| +\mathcal{CW}(\hat{\pi}^N_c, \pi).
    \end{align*}
    As in the proof of Corollary \ref{thm_rates_C_N} we can now use the DKW inequality combined with \eqref{subcube_bound} to obtain $\E[\sup_{y\in \R} |F_{Y^N}(y)-F_Y(y)|]\le cN^{-1/3}$. Lastly, from \eqref{eq: general_inequalities} we obtain $\mathcal{CW}\le \mathcal{KR}$ and $\E[\mathcal{KR}(\pi, \hat{\pi}^N_c)]\le cN^{-1/3}$ follows from Theorem \ref{lem:triangle_est}. This yields
    \begin{align*}
        \E \big[|D(\pi,\hat{\pi}^N_c)|\big] \leq c N^{-1/3}.
    \end{align*}
    Since \(Y\) and \(Y^N\) have a continuous distribution function \(\P\)-almost surely, the denominators of \(\xi_\varphi(\pi)\) and \(\xi_\varphi(\hat{\pi}^N)\) coincide; see \eqref{eq_denom_xi_phi}. This concludes the proof for the continuous adapted empirical measure. The rates based on the empirical checkerboard measure $\hat{\pi}^{N,\#}$ follow similarly.
\end{proof}

By Proposition \ref{prop:comparison} the conclusions of Corollary \ref{cor:strong_consis} and Theorem \ref{cor:rates} remain unchanged if one replaces $\hat{\pi}^N_c$ by $(\hat{\pi}^N_c)^\uparrow$.

\section{Rearranged dependence measures}\label{sec_RDM}

Rearranged dependence measures \cite{strothmann2022} are defined via increasing rearranged copulas applied to a functional \(\eta\) that satisfies axioms \ref{axiom1}--\ref{axiom3} on the class of SI copulas. More
precisely, we assume that \(\eta\colon \cC^\uparrow \to \R\) fulfills
\begin{enumerate}[label = (C\arabic*)]
    \item \label{axiom_eta1} \(\eta(C)\in [0,1]\) for all \(C\in \cC^\uparrow\),
    \item \label{axiom_eta2} \(\eta(C) = 0 \text{ if and only if } C = C^\perp\),
    \item \label{axiom_eta3} \(\eta(C) = 1 \text{ if and only if } C=C^+\).
\end{enumerate}
Estimating rearranged dependence measures also requires continuity of \(\eta\) with respect to pointwise convergence of copulas, i.e.
\begin{enumerate}[label = (C\arabic*)]
\setcounter{enumi}{3}
    \item \label{axiom_eta4} For \(C_n,C\in \cC^\uparrow,\) \(C_n(u,v) \to C(u,v)\) for all \((u,v)\in [0,1]^2\) implies \(\eta(C_n) \to \eta(C)\).
\end{enumerate}
Popular examples of functionals $\eta$ that satisfy the properties \ref{axiom_eta1}--\ref{axiom_eta4} are Spearman's rho, Kendall's tau or Schweizer-Wolff measures; see \cite{strothmann2022}. Further examples are standardized linear functionals of the form \(\pi\mapsto\int g \,\de \pi\) where \(g\colon [0,1]^2\to \R\) is supermodular.\footnote{A function \(g\colon [0,1]^2\to \R\) is said to be supermodular if \(g(u) + g(v) \leq g(u\wedge v) + g(u\vee v)\) for all \(u,v\in [0,1]^2\) where \(\wedge\) and \(\vee\) denote the componentwise minimum and maximum, respectively.}

\begin{definition}\label{def_RDM}
Assume that \(\eta\) satisfies properties \ref{axiom_eta1}--\ref{axiom_eta3} on the class of SI copulas. Then the \emph{rearranged dependence measure} \(\sR_\eta\) induced by $\eta$ is defined as
\begin{align}\label{def_rdm}
    \sR_\eta(\pi) := \eta(C_\pi^\uparrow)\qquad\text{ for all }\pi \in \mathcal{P}_c(\R^2).
\end{align}
\end{definition}

Recall from Lemma \ref{lem_char_incRearr} that \(C_\pi^\uparrow\) is SI, and \(C_\pi^\uparrow = C^\perp\) if and only if \(\pi = \pi_1 \otimes \pi_2\). Further, \(C_\pi^\uparrow = C^+\) if and only if \(\pi = (x,f(x))_\#\pi_1\) for some measurable function \(f\colon \R \to \R\). Hence, \(\sR_\eta\) is a dependence measure that characterizes independence and perfect dependence on the whole class \(\cP_c(\R^2)\) of bivariate distributions having a continuous distribution function:

\begin{lemma}[{\cite[Theorem 2.4]{strothmann2022}}]
    The rearranged dependence measure \(\sR_\eta\) in \eqref{def_rdm} satisfies the axioms \ref{axiom1}--\ref{axiom3} on \(\cP_c(\R^2)\).
\end{lemma}

A natural estimator for the rearranged copula \(C^\uparrow_\pi\) is the increasing rearranged copula $(\hat{C}^\#_{N})^\uparrow$ of a checkerboard approximation of \(C_\pi\), see Section \ref{sec:comp_checkerboard}. 
More precisely, if \(\eta\) satisfies \ref{axiom_eta1}--\ref{axiom_eta4}, then 
\begin{align}
    \hat{\sR}_{\eta}^{N,\#} := \eta((\hat{C}_N^{\#})^\uparrow),
\end{align}
is a strongly consistent estimator of the rearranged dependence measure \(\sR_\eta(\pi)\) in \eqref{def_rdm}; see \cite[Theorem 3.3]{strothmann2022}. 

In this paper, we also consider an alternative estimator based on the continuous adapted empirical measure $\hat{\pi}^N_c$ in \eqref{eq:def_cont}. Let $\pi \in \mathcal{P}_c([0,1]^2)$.
Recall that  
\begin{align}
    \hat{C}_N^\uparrow = C_{\hat{\pi}_c^N}^\uparrow
\end{align}
denotes the increasing rearranged copula associated with \(\hat{\pi}_c^N\); see \eqref{def_Incr_Rearr_Adap_Emp_Cop}. Note that \(\hat{C}_N^\uparrow\) and \((\hat{C}_N^{\#})^\uparrow\) are different estimators of \(C^\uparrow_\pi\), as discussed in Example \ref{ex:checkerboard}.
Based on the increasing rearranged adapted empirical copula, we propose
\begin{align}\label{est_RDM}
    \hat{\sR}_\eta^N := \eta(\hat{C}_N^\uparrow)
\end{align}
as an estimator for \(\sR_\eta(\pi)\). 
Due to Theorem \ref{thm_conv_C_N}\,$(iii)$, we have \(\|\hat{C}_N^\uparrow- C^\uparrow_\pi\|_\infty \to 0\) \(\P\)-almost surely. Since \(\eta\) satisfies properties \ref{axiom_eta1}--\ref{axiom_eta4}, we obtain the following consistency result for our estimator.

\begin{corollary}[Strong consistency of \(\hat{\sR}_\eta^N\)]\label{cor_consistency_rdm}
    For the estimator in \eqref{est_RDM} of the rearranged dependence measure \eqref{def_rdm}, we have
    \begin{align}
        \hat{\sR}_\eta^N \to \sR_\eta(\pi) \qquad \P\text{-almost surely}.
    \end{align}
\end{corollary}
The following result gives convergence speeds for plug-in estimators based on the adapted copula estimator and the checkerboard copula estimator. These rates of convergence resolve an open question in \cite[Section 5]{strothmann2022}.
\begin{theorem}\label{cor:rates_rear}
Assume that $\eta$ satisfies  \ref{axiom_eta1}--\ref{axiom_eta3} and the Lipschitz property
\begin{align}\label{eq:Lipschitz_copula}
    |\eta(C)-\eta(\tilde C)| \le L\cdot d_1(C, \tilde C) \qquad \text{ for all }C, \tilde C\in \mathcal{C}^\uparrow.
\end{align}
If \(\pi\in \cP_c([0,1]^2)\) has Lipschitz kernels \eqref{ass_Lip_kernel} and $y \mapsto F_Y(y)$ is Lipschitz continuous for $Y\sim \pi_2$, then
\begin{align*}
    \E[|\hat{\sR}^N_\eta - \sR_\eta(\pi)|] \le c N^{-1/3}.
\end{align*}
If \(\pi\in \cP_c(\R^2)\) and $\bar\pi = (F_X, F_Y)_{\#}\pi$ has Lipschitz kernels, then 
\begin{align*}
    \E[|\hat{\sR}^{N,\#}_\eta - \sR_\eta(\pi)|] \le c N^{-1/3}.
\end{align*}
\end{theorem}
\begin{proof}
    Using \eqref{eq:Lipschitz_copula} and Corollaries \ref{thm_rates_C_N} and \ref{cor_incrRear} we have 
    \begin{align*}
        \E[|\hat{\sR}^N_\eta - \sR_\eta(\pi)|] \le L\cdot  \E[d_1(C_\pi^\uparrow, \hat{C}_N^\uparrow)] \le c N^{-1/3}
    \end{align*}
    as claimed. The second claim follows similarly from Theorem \ref{thm:rates}.
\end{proof}
\begin{remark}
    In the case of the adapted copula estimator in Theorem \ref{cor:rates_rear}, one can perform a marginal transformation to ensure that \(\tilde \pi = (f,g)_\# \pi\in \cP_c([0,1]^2)\); see the discussion in Remark \ref{rem_ACE_CCE}\,\ref{rem_ACE_CCE2}.
\end{remark}

Let us mention two examples of dependence measures, that satisfy \eqref{eq:Lipschitz_copula}: the first one is Spearman's \(\varrho\), which can be expressed as
\begin{align}
    \varrho(C) = 12 \int_0^1\int_0^1 C(u,v) \,\de u \,\de v - 3;
\end{align}
see e.g.~\cite[Theorem 5.1.6]{Nelsen-2006}.

\begin{corollary}[Spearman's rho]\label{cor_Spearman}
Assume that \(\pi\) has Lipschitz kernels and that $y \mapsto F_Y(y)$ is Lipschitz continuous for $Y\sim \pi_2$.
    Then 
\begin{align*}
    \E[|\hat{\sR}^N_\varrho - \sR_\varrho(\pi)|] \le cN^{-1/3} \quad \text{and} \quad \E[|\hat{\sR}^{N,\#}_\varrho - \sR_\varrho(\pi)|] \le cN^{-1/3}§
\end{align*}    
\end{corollary}
\begin{proof}
For any two copulas $C, \tilde C$ we have
\begin{align*}
|\varrho(C) - \varrho(\tilde C) | &= 12\,  \Big| \int_{[0,1]^2} \big[C(u,v) - \tilde C(u,v) \big]\, \de u \,\de v \Big| \le 12 \|C-\tilde C\|_1 
\stackrel{\eqref{eq:l1_d1}}{\le} 12 d_1(C, \tilde C).
\end{align*}
Now the result follows from Theorem \ref{cor:rates_rear}.
\end{proof}
\begin{figure}[t]
    \centering
    \includegraphics[width=0.95\linewidth]{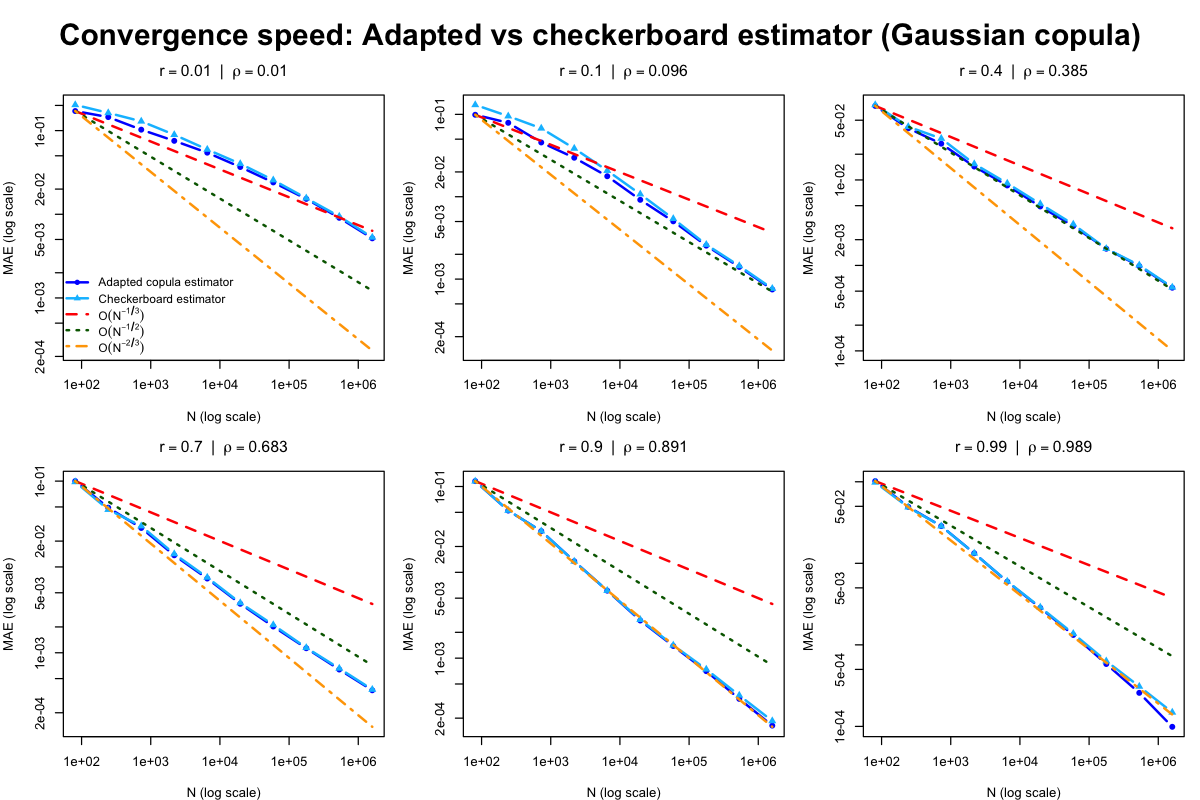}
    \caption{Average rate of convergence of the adapted copula-based estimator \(\hat{\sR}_\varrho = \varrho(\hat{C}_N^\uparrow)\) (dark blue) and the checkerboard copula-based estimator \(\hat{\sR}^\#_\varrho = \varrho((\hat{C}_N^\#)^\uparrow)\) (light blue) to \(\sR_\varrho(\pi) = 6/\pi \arcsin(r/2)\) for samples of size \(N \in \{3^4,3^5,\ldots,3^{13}\}\) from a Gaussian copula \(C_\pi\) with parameter \(r \in \{0.01, 0.1, 0.4, 0.7, 0.9, 0.99\}\). Errors are based on \(500\) runs.
    }
    \label{fig:rates_1}
\end{figure}

\begin{figure}[t]
    \centering
    \includegraphics[width=0.95\linewidth]{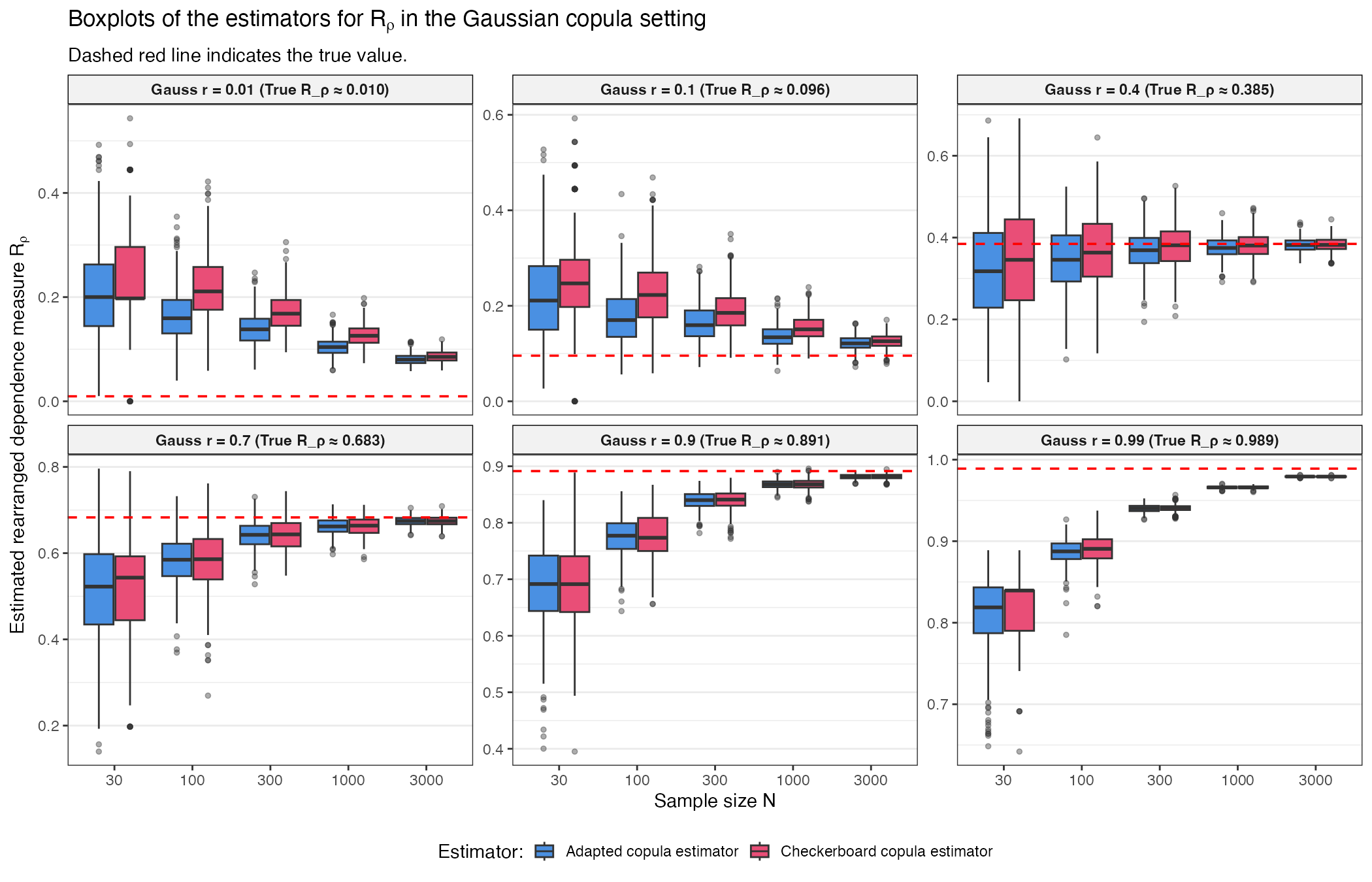}
    \caption{Comparison of the adapted copula-based estimator \(\hat{\sR}_\varrho = \varrho(\hat{C}_N^\uparrow)\) and the checkerboard copula-based estimator \(\hat{\sR}^\#_\varrho = \varrho((\hat{C}_N^\#)^\uparrow)\) for \(\sR_\varrho(\pi)\) in the Gaussian copula setting (i.e. \(C_\pi\) is a Gaussian copula) with parameter \(r\in \{0.01, 0.1, 0.4, 0.7, 0.9, 0.99\}\) for sample sizes \(N\in \{30,100,300,1000,3000\}\). Each boxplot is based on \(500\) runs.
    }
    
    \label{fig:boxplots_1}
\end{figure}
\begin{example}
    Figures~\ref{fig:rates_1} and~\ref{fig:boxplots_1} illustrate the rates of 
convergence for the estimators $\hat{\sR}_\varrho^N$ and 
$\hat{\sR}_\varrho^{N,\#}$ from Corollary~\ref{cor_Spearman} within 
the Gaussian copula setting. More precisely, $\pi$ has a Gaussian 
copula with parameter $r$, for which the rearranged dependence measure admits the 
closed-form expression $\sR_\varrho(\pi) = 6/\pi \arcsin(r/2)$. Figure~\ref{fig:rates_2} illustrates 
the rates of convergence for copulas of $U$-shaped data. In this second 
setting, we assume that $(X,Y)\sim \pi$ follows the model $Y = X^2 + r\varepsilon$, 
where $X \sim \mathcal{U}^{[-1,1]}$ is independent of $\varepsilon \sim \mathcal{N}(0,1)$. 
In both settings, the adapted and the checkerboard copula-based estimators 
perform very similarly, exhibiting convergence rates of at least $O(N^{-1/3})$ 
and often even faster, depending on the strength of functional dependence. We conjecture, that in general the rates vary between \(N^{-1/3}\) and \(N^{-2/3}\), where \(N^{-1/3}\) is optimal for independence, and \(N^{-2/3}\) for complete dependence. We leave this for future research.
\end{example}

Our second example is the \emph{Schweizer-Wolff measure}  \(\sigma_p\), which is defined as
\begin{align}
    \sigma_p(C) := \frac{\lVert C - C^\perp \rVert_p}{\lVert C^+ - C^\perp \rVert_p}
\end{align}
for $p\ge 1$; see \cite{Schweizer-Wolff-1981}. Here, \(\lVert\cdot\rVert_p\) denotes the \(L^p\)-norm with respect to the Lebesgue measure on \(\R^2\). As shown in \cite[Example 2.6]{strothmann2022}, \(\sigma_p\) is a measure of association that satisfies properties \ref{axiom_eta1}--\ref{axiom_eta4}.

\begin{figure}[t]
    \centering
    \includegraphics[width=0.95\linewidth]{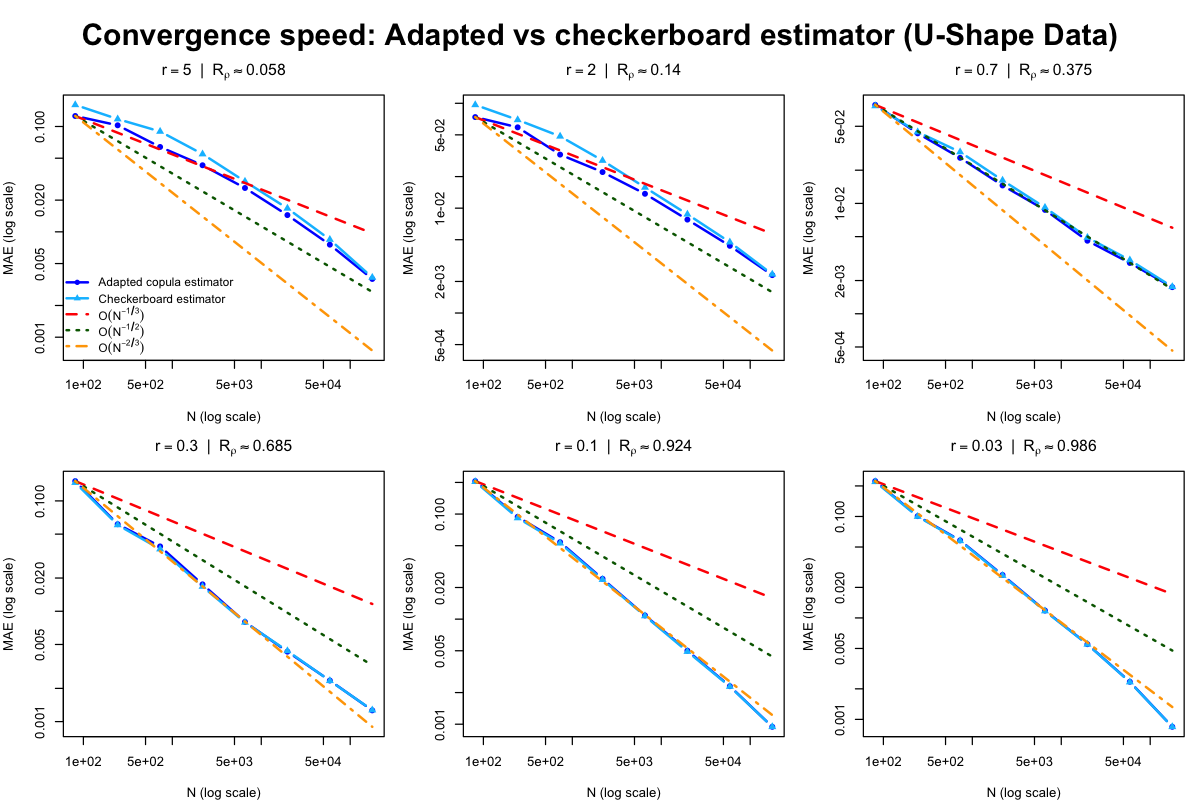}
    \caption{Average rate of convergence of the adapted copula-based estimator \(\hat{\sR}_\varrho = \varrho(\hat{C}_N^\uparrow)\) (dark blue) and the checkerboard copula-based estimator \(\hat{\sR}^\#_\varrho = \varrho((\hat{C}_N^\#)^\uparrow)\) (light blue) for \(\sR_\varrho(\pi)\) for samples of size \(N \in \{3^4,3^5,\ldots,3^{11}\}\) from the copula \(C_\pi\) where \((X,Y)\sim \pi\) with \(Y = X^2 + r\varepsilon\) has \(U\)-shape: \(X\) is uniform on \([-1,1]\) and independent of standard normal \(\varepsilon\), and  \(r \in \{5,2,0.7,0.3,0.1,0.03\}\). The approximate value of \(\sR_\varrho(\pi)\) is based on a sample of size \(1.000.000\).
    }
    \label{fig:rates_2}
\end{figure}

\begin{corollary}[Schweizer-Wolff measure]\label{cor_schweizer_wolff}
Assume that \(\pi\) has Lipschitz kernels and $y \mapsto F_Y(y)$ is Lipschitz continuous for $Y\sim \pi_2$.
Then $$\E[|\hat{\sR}^N_{\sigma_p} - \sR_{\sigma_p}(\pi)|] \le c N^{-1/(3p)}.$$
\end{corollary}

\begin{proof}
For any two copulas $C, \tilde C$ we have 
\begin{align}
    \big|\big\|C-C^\perp\|_p -\|\tilde C- C^\perp\|_p\big|\big| \le \|C-\tilde C\|_p\le (\|C-\tilde C\|_1)^{1/p}\stackrel{\eqref{eq:l1_d1}}{\le} d_1(C, \tilde C)^{1/p},
\end{align}
where we use for the second inequality that $0\le C(u,v), \tilde C(u,v) \le 1$ for all $u,v\in [0,1]^2$. The claim now follows from Corollaries \ref{thm_rates_C_N}, \ref{cor_incrRear} and Jensen's inequality.
\end{proof}

%\textcolor{red}{Should we add a remark, that we are not sure about Kendall's tau but think it's unlikely to hold using the same methodolgy? Or should we not mention it at all?}

\section*{Acknowledgement}

The first author was funded in whole by the Austrian Science Fund (FWF) 
{[10.55776/PAT1669224]} project \emph{SORT: Stochastic orders for functional dependence}. The second author would like to thank NNF and Villum Fonden for their support.

\bibliographystyle{apalike} % Style BST file (imsart-number.bst or imsart-nameyear.bst)
\bibliography{ConditionalWassersteinDistance.bib}

\appendix
\section{Auxiliary results}

\begin{lemma}[{\cite[Corollary 4.2.3]{fdsempi2016}}] \label{lem:polya}
Let $\pi^n,\pi \in \mathcal{P}(\R^{k}\times \R^l)$ have cdfs denoted by $H^n,H$. If $(\pi^n)$ converges weakly to $\pi$ and $H$ is continuous, then $\|H^n-H\|_\infty \to 0$.
Furthermore, if $k=l=1$ and $H^n$ are continuous cdfs, then also the induced copulas  $(C_{\pi^n})$ converge to the induced copula $C_{\pi}$ uniformly.
\end{lemma}

\begin{lemma}[{\cite[Remark 2.19\,(iii)]{Villani-2003}}]\label{lem:L1_integral}
    For distribution functions \(F\) and \(G\) on \(\R\), it is
    \begin{align*}%\label{eq_lem_L1_integral}
        \int_0^1 \left|F^{-1}(u) - G^{-1}(u)\right| \de u = \int_\R \left| F(x) - G(x) \right|\de x.
    \end{align*}
\end{lemma}

\begin{lemma}\label{lem:KR_estimate_y}
    Let $\pi \in \mathcal{P}([0,1]^2)$, let $f:[0,1]\to [0,1]$ and define $\tilde\pi=(x,f(y))_{\#}\pi$. Then 
    \begin{align*}
        \mathcal{KR}(\pi, \tilde \pi) \le \sup_{x\in [0,1]} |f(x)-x|.
    \end{align*}
\end{lemma}

\begin{proof}
    Let $(X,Y)\sim \pi$ and $(\widetilde X, \widetilde Y)\sim \tilde \pi$. By definition we have $X\stackrel{d}{=}\widetilde X$ and $(\widetilde Y| \widetilde X=x)\stackrel{d}{=} (f(Y)|X=x)$. Thus we note that
    \begin{align}\label{eq:KR_split}
    \begin{split}
         \mathcal{KR}(\pi, \tilde\pi) &= \int_0^1 |F_X^{-1}(u) - F_{\widetilde X}^{-1}(u)| + \int_0^1 |F^{-1}_{Y|X=F_{X}^{-1}(u)}(v)-F^{-1}_{\widetilde Y|\widetilde X= F_{\widetilde X}^{-1}(u)}(v)|\,\de v \de u\\
         &= \int_{[0,1]^2} |F^{-1}_{Y|X=F_{X}^{-1}(u)}(v)-F^{-1}_{\widetilde Y|\widetilde X= F_{X}^{-1}(u)}(v)|\,\de v \de u\\
         &= \int_0^1 \mathcal{W}(\pi_{F_X^{-1}(u)}, \tilde \pi_{F_{X}^{-1}(u)})\,\de u\\
         &= \int_0^1 \mathcal{W}(\pi_{F_X^{-1}(u)},f_{\#}\pi_{F_{ X}^{-1}(u)})\,\de u\\
         &\le \sup_{x\in [0,1]} |f(x)-x|.
    \end{split}
    \end{align}
This shows the claim.
\end{proof}

\begin{lemma}\label{lem:general_marginals}
Let $g, \tilde g :\R\to \R$ be Borel measurable, and assume that there exists a modulus of continuity $\omega_g$ for $g$ such that 
\begin{align*}
   |g(y)- g(\tilde y)|\le \omega_g(|y-\tilde y|)
\end{align*}
for all $y,\tilde y\in \R$. Then for $\pi, \tilde \pi \in \mathcal{P}(\R\times \R)$ with continuous first marginals we have for $X\sim \pi_1, \widetilde X\sim \tilde \pi_1$,
\begin{align}\label{eq:kr_app1}
\begin{split}
    \mathcal{KR}((F_X,g)_{\#}\pi, (F_{\widetilde X}, \tilde g)_{\#} \tilde\pi) &\le  \int_0^1\inf_{\gamma_u \in  \Pi( \pi_{F_X^{-1}(u)},  \tilde \pi_{F_{\tilde X}^{-1}(u)}) } \int \omega_g( |y-\tilde y|)\,\gamma_u(\de y, \de \tilde y)\, \de u +\|g- \tilde g\|_\infty.
\end{split}
\end{align}
If furthermore $\omega_g$ is concave, then 
\begin{align}\label{eq:kr_app2}
\begin{split}
    \mathcal{KR}((F_X,g)_{\#}\pi, (F_{\widetilde  X}, \tilde g)_{\#} \tilde\pi) &\le \omega_g  \Big( \int_0^1 \mathcal{W}\big( \pi_{F_X^{-1}(u)},  \tilde \pi_{F_{\tilde X}^{-1}(u)}  \big)\,\de u\Big) + \|g- \tilde g\|_\infty\\
    &\le \omega_g\big(\mathcal{KR}(\pi, \tilde \pi)\big) + \|g- \tilde g\|_\infty.
\end{split}
\end{align}
\end{lemma}

\begin{proof}
    As $({F_X})_{\#}\pi_1= (F_{\tilde X})_{\#}\tilde \pi_1 =\mathcal{U}$ and $F_X, F_{\tilde X}$ are continuous, we obtain 
    \begin{align*}
        \big(F_X,g)_{\#}\pi\big)_u = g_{\#}\pi_{F_X^{-1}(u)}, \quad \big(F_{\tilde X},\tilde g)_{\#}\tilde \pi\big)_u =  \tilde g_{\#}\tilde \pi_{F_{\tilde X}^{-1}(u)}
    \end{align*}
    for $\mathcal{U}$-a.e.~$u\in [0,1].$
    Thus, from the definition
    \begin{align*}
        \mathcal{KR}((F_X,g)_{\#}\pi, (F_{\tilde X}, \tilde g)_{\#} \tilde\pi) &= \int_0^1 \mathcal{W}\big(g_{\#} \pi_{F_X^{-1}(u)}, \tilde g_{\#} \tilde \pi_{F_{\tilde X}^{-1}(u)}\big) \, \de u.
    \end{align*}
    For fixed $u\in [0,1]$,
    \begin{align*}
        \mathcal{W}\big(g_{\#} \pi_{F_X^{-1}(u)}, \tilde g_{\#} \tilde \pi_{ F_{\tilde X}^{-1}(u)}\big)
        &\le \mathcal{W}\big(g_{\#} \pi_{F_X^{-1}(u)},  g_{\#} \tilde \pi_{F_{\tilde X}^{-1}(u)}\big) + \mathcal{W}\big(g_{\#} \tilde \pi_{F_{\tilde X}^{-1}(u)}, \tilde g_{\#} \tilde  \pi_{F_{\tilde X}^{-1}(u)}\big)\\
        &\le \mathcal{W}\big(g_{\#} \pi_{F_X^{-1}(u)},  g_{\#} \tilde \pi_{F_{\tilde X}^{-1}(u)}\big) + \|g-\tilde g\|_\infty.
    \end{align*}
    and for any $\gamma_u\in \Pi\big( \pi_{F_X^{-1}(u)},  \tilde \pi_{F_{\tilde X}^{-1}(u)}\big)$ we have
    \begin{align*}
        \mathcal{W}\big(g_{\#} \pi_{F_X^{-1}(u)},  g_{\#} \tilde \pi_{F_{\tilde X}^{-1}(u)}\big) \le \int |g(y)- g(\tilde y)| \,\gamma_u(\de y, \de \tilde y) \le \int \omega_g( |y-\tilde y|) \, \gamma_u(\de y, \de \tilde y).
    \end{align*}
    Combining these estimates we obtain \eqref{eq:kr_app1}. Applying Jensen's inequality to the last integral, we obtain \eqref{eq:kr_app2}. 
\end{proof}

\section{Remaining proofs}

\begin{lemma} \label{lem: kpc}
For all $\pi, \tilde \pi\in \mathcal{P}(\mathcal{H}\times \mathcal{H})$ we have
    \begin{align*}
   \big| \mathrm{MMD}^2(\pi_x, \pi_2) - \mathrm{MMD}^2(\tilde \pi_{\tilde{x}}, \tilde \pi_2) \big| \le c \cdot [\mathcal{W}(\pi_x,\tilde \pi_{\tilde x}) + \mathcal{W}(\pi_2, \tilde \pi_2) \big]. 
   \end{align*}
\end{lemma}

\begin{proof}
Recall that for all $\mu,\nu \in \mathcal{P}(\mathcal{H})$ there exist $\mu_{\R^d},$ $\nu_{\R^d}\in \mathcal{P}(\R^d)$ such that we have
\begin{align*}
    \text{MMD}(\mu,\nu)=  \Big\| \int \phi(x) \,\mu_{\R^d}(\de x) -\int \phi(x)\,\nu_{\R^d}(\de x)\Big\|_{\mathcal{H}}.
\end{align*} 
From \cite[Proof of Lemma 5.4]{han2026max} it follows that that
\begin{align*}
    |\mathrm{MMD}(\mu,\nu) - \mathrm{MMD}(\tilde \mu, \tilde \nu)|\le \mathcal{W}(\mu, \tilde\mu) + \mathcal{W}(\nu, \tilde\nu)
\end{align*}
for all $\mu, \nu, \tilde{\mu}, \tilde{\nu} \in \mathcal{P}(\mathcal{H})$. 
Boundedness of $\phi$ implies boundedness of $\mathrm{MMD}$, so that
\begin{align*}
   \big| \text{MMD}^2(\pi_x, \pi_2) - \text{MMD}^2(\tilde \pi_{\tilde x}, \tilde \pi_2)\big| \le c [\mathcal{W}(\pi_x,\tilde \pi_{\tilde x})+ \mathcal{W}(\pi_2, \tilde \pi_2)]. 
\end{align*}
\end{proof}

\begin{proof}[Proof of Lemma \ref{lem_char_incRearr}]
    \ref{lem_char_incRearr1}: If \(\pi = \pi_1\otimes \pi_2\), then the induced copula is the independence copula, i.e. \(C_\pi = C^\perp\). Since \(C^\perp\) is SI, it follows that \(C_\pi^\uparrow = C^\perp\). Conversely, assume \(C_\pi^\uparrow = C^\perp\) and thus \(C_{\pi^\uparrow} = C^\perp\). Then, for \((X,Y)\sim \pi\) and \((X^\uparrow,Y^\uparrow)\sim \pi^\uparrow\), the conditional probabilities 
    \begin{align*}
        \pi_X((-\infty,y]) = F_{Y|X}(y) \eqd F_{Y^\uparrow|X^\uparrow}(y) = \pi^\uparrow_{X^\uparrow}((-\infty,y]) \stackrel{\eqref{eq:d1}}{=} \partial_1 C_{\pi^\uparrow}(F_{X^\uparrow}(x),F_{Y^\uparrow}(y)) = F_{Y^\uparrow}(y) = F_Y(y)
    \end{align*}
     are constant for all \(y\in \R\). Hence, \(X\) and \(Y\) are independent and thus \(\pi = \pi_1\otimes \pi_2\).\\
    \ref{lem_char_incRearr2}: For \((X,Y)\sim \pi\), \(\pi = (x,f(x))_\# \pi_1\) is equivalent to \(Y=f(X)\) \(\P\)-almost surely. The latter is equivalent to \(\partial_1 C_\pi(u,v) \in \{0,1\}\) for Lebesgue-almost all \((u,v)\in [0,1]^2\); see \cite[Proposition 3]{siburg2013}. This means for \((X^\uparrow,Y^\uparrow)\sim \pi^\uparrow\) that \(\partial_1 C_\pi(F_X(X),F_Y(y)) = \pi_X([0,y]) \eqd \pi_{X^\uparrow}^\uparrow([0,y]) = \partial_1 C_{\pi^\uparrow}(F_{X^\uparrow}(X^\uparrow),F_Y(y))\in \{0,1\}\) for all \(y\) \(\P\)-almost surely. The latter, however, is equivalent to the property that \(\partial_1 C_\pi^\uparrow(u,v)\in \{0,1\}\) for $\mathcal{U}\otimes\mathcal{U}$-a.e.~$(u,v)\in [0,1]^2$. The only SI copula that satisfies this property is the upper Fr\'{e}chet copula \(C^+\).\\
    \ref{lem_char_incRearr3}: Since \(C_\pi^\uparrow\) is SI, the statement follows from the following well-known identity for SI copulas (see e.g.~\cite[Figure 3.2]{Muller-2002}):
For \(C\in \cC^\uparrow\), we have \(C(u,v) = \int_0^u \partial_1 C(t,v) \de t \geq uv = C^\perp(u,v)\) using that \(x\mapsto \partial_1 C(x,v)\) is decreasing, and $\int_0^1 \partial_1 C(t, v)\,\de t=C(1,v)-C(0,v)=v$. The inequality \(C(u,v)\leq C^+(u,v)\) is true for every copula; see e.g.~\cite[2.10.12]{Nelsen-2006}.
\end{proof}

\end{document}